\documentclass[11pt]{article}

\usepackage{subfigure}
\usepackage{color}
\usepackage[cmex10]{amsmath}
\usepackage{mathrsfs,amssymb,amsmath}
\usepackage{cases}
\usepackage{graphicx}
\usepackage{caption}

\usepackage{epstopdf}
\usepackage{hyperref}

\usepackage{float}
\usepackage{placeins}

\newcommand{\R}{{\mathbb R}}

\def\dref#1{(\ref{#1})}
\def\crr{\cr\noalign{\vskip-0.0mm}}

\usepackage{algorithm} 
\usepackage{algorithmic} 

\newtheorem{corollary}{Corollary}[section]
\newtheorem{remark}{Remark}[section]
\newtheorem{lemma}{Lemma}[section]

\newtheorem{proposition}{Proposition}

\newtheorem{theorem}{Theorem}[section]
\newtheorem{definition}{Definition}[section]

\newenvironment{proof}{{\bf {Proof.}}}{\hfill $\square$}
\allowdisplaybreaks[4]

\numberwithin{equation}{section}

\def\dref#1{(\ref{#1})}

\def\pt{\partial}

\def\e{\varepsilon}

\def\vp{\varphi}

\def\bf{\textbf}
\def\pt{\partial}

\def\Om{\Omega}

\def\al{\alpha}

\def\de{\delta}

\def\De{\Delta}

\def\iy{\infty}

\def\f{\frac}

\def\se{\setminus}

\def\df{\mathrm d}

\def\wt{\widetilde}
\def\wh{\widehat}

\def\essinf{\operatorname*{ess\ \! inf}}

\def\hra{\hookrightarrow}

	\DeclareMathOperator{\Div}{div}

	\DeclareMathOperator{\dist}{dist}

	\DeclareMathOperator{\supp}{{supp}}

	\newcommand{\N}{\mathbb N}

\begin{document}
	
	
	\title{\bf {\bf  Weak  Unique Continuation  on Annular Domains for Backward Degenerate Parabolic Equations with Degenerate Interior Points
}}
	\author{Dong-Hui Yang, Bao-Zhu Guo, Guojie Zheng, Jie Zhong}

\title{{\bf   {\bf Quantitative   Weak  Unique Continuation  on Annular Domains for Backward Degenerate Parabolic Equations with Degenerate Interior Points}}\footnote{\small This work was carried out with the support of the
National Natural Science Foundation of China under grant  nos. 12131008 and U23B2033, and
National Key R\&D Program of China under grant no. 2024YFA1013101.}}

\author{ Dong-Hui Yang$^{a}$\footnote{\small
The corresponding author. Email: donghyang@139.com},  Bao-Zhu Guo$^{b}$, Guojie Zheng$^{c}$,  Jie Zhong$^{d}$
\\
$^a${\it School of Mathematics and Statistics, Central South University}\\
			{\it Changsha 410075, P.R.China}\\
$^b${\it Academy of Mathematics and Systems Science, Academia Sinica, Beijing 100190, China}\\
$^c${\it College of Digital Technology and Engineering}\\
 {\it Ningbo University of Finance and Economics, Ningbo 315175, China} \\
$^d${\it   Department of Mathematics}\\
{\it California State University Los Angeles, Los Angeles, 90032, USA}
}
\date{}

	\maketitle{}
	\thispagestyle{empty}
	\thispagestyle{empty}
	
	
	\begin{abstract}
		  In this paper, we establish  a quantitative weak unique continuation theorem on an annular domain for a backward degenerate parabolic equation with a degenerate interior point. Our methodology hinges on approximating the solution of the degenerate parabolic equation through solutions of non-degenerate parabolic counterparts. Subsequently, we establish Carleman estimates for the non-degenerate parabolic equation across two separate domains. By virtue of these estimates, we deduce a quantitative weak unique continuation property for the degenerate parabolic equation, thereby substantiating the weak unique continuation result for the original degenerate parabolic equation.
	\end{abstract}
	
	\section{Introduction}
	
	 The unique continuation property for elliptic equations has been extensively explored in numerous studies, including \cite{Adolfsson,Brkri,Carleman,Garofalo2,Hormander,Kenig,Koch,Kukavica1,Kukavica2,Vessella1}. Carleman \cite{Carleman} pioneered the application of Carleman estimates to establish unique continuation, a method subsequently extended by H\"{o}rmander \cite{Hormander} and others. A comprehensive review of unique continuation can be found in \cite{Kenig}. Additionally, several papers have investigated unique continuation under the doubling balls or three balls condition, as evidenced in \cite{Kukavica1,Kukavica2}.
Similarly, unique continuation for parabolic equations has been investigated in various studies, such as \cite{Canuto,Escauriaza,FG,Lin,Phung,Vessella}. Carleman estimates remain a powerful technique for proving unique continuation in this context \cite{Canuto,Escauriaza,FG,Vessella}.
However, relatively few papers address unique continuation for degenerate elliptic and parabolic equations, and these typically focus on specific degenerate partial differential equations, e.g., \cite{Alabau,Banerjee,Garofalo,Wu}.
In this paper, we investigate a form of weak unique continuation for degenerate parabolic equations with degenerate interior points using Carleman estimates and approximation methods. The core concepts are as follows: a) Initially, approximate the solution of the degenerate parabolic equation through solutions of non-degenerate parabolic equations; b) Subsequently, derive Carleman estimates for non-degenerate equations; and c) Ultimately, utilize these estimates to approximate those for the degenerate parabolic equation and attain weak unique continuation. Nevertheless, this paper only considers weak unique continuation on certain annular domains; extending this to arbitrary domains remains future work. Utilizing an approximation method, we have previously obtained weak unique continuation for degenerate elliptic equations with degenerate interior points \cite{Wu1}; for more approximation techniques, we refer to \cite{Cora}.

The equation examined in this paper is the following backward degenerate parabolic equation featuring a degenerate interior point:
\begin{equation}\label{12.14.1}
\begin{cases}
\partial_t \varphi + \Div(|x|^\alpha \nabla \varphi) = 0, & \text{ in } Q, \\
\varphi = 0, & \text{ on } \partial Q, \\
\varphi(T) = \varphi_T, & \text{ in } \Omega,
\end{cases}
\end{equation}
where $\alpha \in (0,2)$ is a given constant, $\Omega \subset \mathbb{R}^N$ ($N \geq 2$) is a bounded domain containing $0$ with $\partial \Omega \in C^2$, $Q = \Omega \times (0,T)$ with $T > 0$ being a constant, and $\varphi_T \in L^2(\Omega)$ is the terminal data. We recall that the solution $\varphi$ of \eqref{12.14.1} satisfies the weak unique continuation property:
\begin{equation}\label{04.18.1}
\text{ If } \varphi = 0 \text{  on } \omega \times (0,T), \text{  then } \varphi = 0 \text{  on } \Omega.
\end{equation}
For solutions of degenerate parabolic equations with $A_p$ weight coefficients, a good approximation by $A_p$ weights has been considered in \cite[lemma 2,section 2]{FF}, or in \cite{CS2,CS3}. However, these approximate $A_p$ weights are merely Lebesgue measurable functions and may lack partial derivatives. To achieve weak unique continuation for degenerate parabolic equations with degenerate interior points, we require the approximate $A_p$ weights to be differentiable, necessitating a different approximation approach, which is a key contribution of this paper.
This paper focuses on dimensions $N \geq 2$, as the case $N = 1$ has been extensively studied, e.g., in \cite{Alabau}. We also consider annular domains since domains $\omega$ containing $0$ have been examined in \cite{Wu}.

The paper proceeds as follows: in Section \ref{Section 2}, we provide solutions for the degenerate parabolic equation \eqref{12.14.1}. In Section \ref{Section 3}, we identify an approximate non-degenerate parabolic equation for \eqref{12.14.1}. Sections \ref{Section 4} and \ref{Section 5} present two types of Carleman estimates, one within $\omega$ and the other outside $\omega$, both controllable by the part on $\omega$. Finally, in Section \ref{Section 6}, we establish the weak unique continuation \eqref{04.18.1} for the equation \eqref{12.14.1}.

 \section{Solution spaces}\label{Section 2}

In this section, we shall give the solution spaces of degenerate parabolic equation \eqref{12.14.1}.

Let $B_R = \{x \in \mathbb{R}^N \colon |x| < R\}$, and
\begin{equation*}
w = |x|^\alpha.
\end{equation*}
Define
\begin{equation*}
\Omega = A_{3R, 6R} = \{x \in \mathbb{R}^N \colon 3R < |x| < 6R\},
\end{equation*}
which is an open annulus subset with $0 < 3R < 8R < \mathrm{dist}(0, \partial \Omega)$.
It is evident that $w$ is an $A_{1 + \frac{2}{N}}$ weight since $w$ is an $A_p$ weight if and only if $-N < \alpha < N(p - 1)$ (see e.g., \cite[p.10]{Heinonen}, or   \cite[chapter IV]{Garcia-Cuerva}).

 We recall that for $1 < p < \infty$, a locally integrable, non-negative function $w$ is termed an $A_p$ weight if there exists a constant $C > 0$ such that, for all cubes $K$ in $\mathbb{R}^N$, the following inequality holds:
\begin{equation}\label{04.24.1}
\left(\frac{1}{|K|}\int_K w(x)\, \mathrm{d}x\right)\left(\frac{1}{|K|}\int_K w(x)^{-\frac{1}{p-1}}\, \mathrm{d}x\right)^{p-1} \leq C.
\end{equation}
The infimum $C(w,p)$ of the set of all positive constants $C$ satisfying inequality \eqref{04.24.1} is referred to as the $A_p$ constant of $w$. When $p = 1$, a locally integrable, non-negative function $w$ is said to be an $A_1$ weight if there exists a constant $C > 0$ such that, for all cubes $K$ in $\mathbb{R}^N$,
\begin{equation*}
\frac{1}{|K|}\int_K w(x)\, \mathrm{d}x \leq C \cdot \essinf_{x \in K} w(x).
\end{equation*}
For further references on $A_p$ weights and the $A_p$ constant of $w$, the reader is referred  to \cite{FF, Garcia-Cuerva, Heinonen}.
In what follows, we shall denote
\begin{equation*}
w(E) = \int_E w \, dx,
\end{equation*}
where $E \subset \Omega$ is a Lebesgue measurable set.
We define
\begin{equation*}
L^2(\Omega; w) = \left\{u \in \mathcal{D}'(\Omega) \colon \int_\Omega u^2 w \, dx < \infty\right\},
\end{equation*}
where $\mathcal{D}'(\Omega)$ is the distribution space and $\mathcal{D}(\Omega) = C_0^\infty(\Omega)$.
The inner product on $L^2(\Omega; w)$ is
\begin{equation*}
(u, v)_{L^2(\Omega; w)} = \int_\Omega uvw \, dx,
\end{equation*}
and the norm on $L^2(\Omega; w)$ is
\begin{equation*}
\|u\|_{L^2(\Omega; w)} = \left(\int_\Omega u^2 w \, dx\right)^{\frac{1}{2}}.
\end{equation*}
It is well known that $(L^2(\Omega; w), (\cdot, \cdot){L^2(\Omega; w)})$ is a Hilbert space and $(L^2(\Omega; w), \|\cdot\|{L^2(\Omega; w)})$ is a Banach space (see, e.g,\cite[chapter 1]{Heinonen}, or  \cite[chapter IV]{Garcia-Cuerva}).
Set
\begin{equation*}
H^1(\Omega; w) = \left\{u \in L^2(\Omega; w) \colon \frac{\partial u}{\partial x_i} \in L^2(\Omega; w), i = 1, \cdots, N\right\},
\end{equation*}
where $\frac{\partial u}{\partial x_i}$ is the distribution derivatives with respect to the space variables $x_i$ for
$i=1,2,\cdots, N$. The inner product on $H^1(\Omega; w)$ is
\begin{equation*}
(u, v)_{H^1(\Omega; w)} = \int_\Omega uvw \, dx + \sum_{i=1}^N \int_\Omega \frac{\partial u}{\partial x_i} \frac{\partial v}{\partial x_i} w \, dx,
\end{equation*}
and the norm is
\begin{equation*}
\|u\|_{H^1(\Omega; w)} = \left(\int_\Omega u^2 w \, dx + \sum_{i=1}^N \int_\Omega \left|\frac{\partial u}{\partial x_i}\right|^2 w \, dx\right)^{\frac{1}{2}}.
\end{equation*}
Define
\begin{equation*}
H_0^1(\Omega; w) = \text{ the closure of } \mathcal{D}(\Omega) \text{  in } H^1(\Omega; w).
\end{equation*}
Let $H^{-1}(\Omega; w)$ denote the dual space of $H_0^1(\Omega; w)$. It is clear that $H^1(\Omega; w)$ is a subspace of $\mathcal{D}'(\Omega)$, and it is also well known that $(H^1(\Omega; w), (\cdot, \cdot)_{H^1(\Omega; w)})$ is a Hilbert space and $(H^1(\Omega; w), \|\cdot\|_{H^1(\Omega; w)})$ is a Banach space (see, e.g., \cite[chapter 1]{Heinonen}, or  \cite[chapter IV]{Garcia-Cuerva}).

  The following Lemma \ref{08.16.L3} has been documented in \cite[theorem 1.3]{Fabes}, \cite[section 1.4, p.9]{Heinonen}, and \cite{Trudinger}. Nevertheless, for the sake of completeness within our paper, we re-establish Lemma \ref{08.16.L3} as Lemma \ref{08.15.L1}.
\begin{lemma}\label{08.16.L3}
Let $\Omega \subset \mathbb{R}^N$ represent an open, bounded domain. Assume $w \in A_p$ for $1 < p < \infty$. Consequently, there exist constants $C_\Omega$ and $\delta > 0$ such that for all $u \in C_0^\infty(\Omega)$ and all $k \in \left[1, \frac{N}{N - 1} + \delta\right]$,
\begin{equation*}
\|u\|_{L^{kp}(\Omega; w)} \leq C_\Omega \|\nabla u\|_{L^p(\Omega; w)},
\end{equation*}
where $C_\Omega$ depends exclusively on $N$, the $A_p$ constant of $w$, $p$, and the diameter of $\Omega$.
\end{lemma}
\begin{lemma}\label{08.15.L1}
For every $N \geq 2$ and $\alpha \in (0, 2)$, it holds that
\begin{equation*}
(N - 2 + \alpha)\left\||x|^{\frac{\alpha}{2} - 1}u\right\|_{L^2(\Omega)} \leq 2\|\nabla u\|_{L^2(\Omega; w)}
\end{equation*}
for all $u \in H_0^1(\Omega; w)$. Furthermore, if $u \in H_0^1(\Omega; w)$, then $u \in L^2(\Omega)$.
\end{lemma}
 \begin{proof}
The first part of the proof is derived from \cite[lemma 3.1]{Wu} or \cite[proposition 2.1(1)]{Stuart}.
For the second part, utilizing Lemma \ref{08.15.L1} and the condition $\alpha \in (0, 2)$, we deduce that
\begin{equation}\label{12.09.2}
\frac{N - 2 + \alpha}{2m}\|u\|_{L^2(\Omega; w)} \leq \|\nabla u\|_{L^2(\Omega; w)}   \text{ with } m := \sup_{x \in \Omega}|x| + 1
\end{equation}
and
\begin{equation}\label{08.15.10}
\frac{N - 2 + \alpha}{2m^{1 - \frac{\alpha}{2}}}\|u\|_{L^2(\Omega)} \leq \|\nabla u\|_{L^2(\Omega; w)},
\end{equation}
which is the Poincar\'{e} inequality. From this, we obtain that $u \in L^2(\Omega)$ if $u \in H_0^1(\Omega; w)$.
\end{proof}
\begin{remark}\label{08.16.R1}
If $u \in H_0^1(\Omega; w)$, from Lemma \ref{08.16.L3}, taking $k = 1, p = 2$, we have
\begin{equation}\label{08.16.8}
\|u\|_{L^2(\Omega; w)} \leq C\|\nabla u\|_{L^2(\Omega; w)},
\end{equation}
where the constant $C > 0$ is independent of $u \in H_0^1(\Omega; w)$. This is the Poincar\'{e} inequality in weighted Sobolev space (see, e.g., \cite[chapter 1 in]{Heinonen}).
In particular, the norm
\begin{equation}\label{08.19.2}
\|u\|_{H_0^1(\Omega; w)} = \left(\int_\Omega (\nabla u \cdot \nabla u)w \, dx\right)^{\frac{1}{2}}
\end{equation}
is an equivalent norm in $H_0^1(\Omega; w)$, where $\nabla u = \left(\frac{\partial u}{\partial x_1}, \cdots, \frac{\partial u}{\partial x_N}\right)$ is the gradient of $u$. Here and below, we use \eqref{08.19.2} to define the norm of $H_0^1(\Omega; w)$.
\end{remark}

The Lemma \ref{08.15.L4} following is documented  in \cite[theorem 3.4]{Wu}, or   \cite[proposition 2.1 (5)]{Stuart}.
\begin{lemma}\label{08.15.L4}
The embedding $H_0^1(\Omega; w) \hookrightarrow L^2(\Omega)$ is compact.
\end{lemma}

 We define
\begin{equation*}
L^2(Q; w) = \left\{\varphi \in \mathcal{D}'(Q) \colon \iint_Q \varphi^2 w \, dx \, dt < +\infty\right\},
\end{equation*}
where $\mathcal{D}'(Q)$ denotes the space of distributions on $Q$   and its inner product is given by
\begin{equation*}
(\varphi, \psi)_{L^2(Q; w)} = \iint_Q \varphi \psi w \, dx \, dt.
\end{equation*}
 Similarly, we can introduce the space
\begin{equation*}
L^2(Q; w^{-1}) = \left\{\varphi \in \mathcal{D}'(Q) \colon \iint_Q \varphi^2 w^{-1} \, dx \, dt < +\infty\right\},
\end{equation*}
with its inner product defined as
\begin{equation*}
(\varphi, \psi)_{L^2(Q; w^{-1})} = \iint_Q \varphi \psi w^{-1} \, dx \, dt.
\end{equation*}
Analogous to the case of $L^2(\Omega; w)$, the spaces $(L^2(Q; w), (\cdot, \cdot)_{L^2(Q; w)})$ and $(L^2(Q; w^{-1}), (\cdot, \cdot)_{L^2(Q; w^{-1})})$ are Hilbert spaces.
Next, we define
\begin{equation*}
L^2(0, T; H^1(\Omega; w)) = \left\{\varphi \in L^2(Q; w) \colon \frac{\partial \varphi}{\partial x_i} \in L^2(Q; w), \text{  for } i = 1, \cdots, N\right\},
\end{equation*}
with its inner product given by
\begin{equation*}
(\varphi, \psi)_{L^2(0, T; H^1(\Omega; w))} = \iint_Q \varphi \psi w \, dx \, dt + \iint_Q (\nabla \varphi \cdot \nabla \psi)w \, dx \, dt,
\end{equation*}
and its norm defined as
\begin{equation*}
\|\varphi\|_{L^2(0, T; H^1(\Omega; w))}^2 = \iint_Q \varphi^2 w \, dx \, dt + \iint_Q |\nabla \varphi|^2 w \, dx \, dt.
\end{equation*}
It follows that $(L^2(0, T; H^1(\Omega; w)), (\cdot, \cdot)_{L^2(0, T; H^1(\Omega; w))})$ is a Hilbert space, and $(L^2(0, T; H_w^1(Q)), \|\cdot\|_{L^2(0, T; H_w^1(Q))})$ is a Banach space (\cite[section 3]{FF}). The same conclusions apply to $L^2(0, T; H_0^1(\Omega; w))$.
\begin{remark}
By \eqref{08.16.8} in Remark \ref{08.16.R1}, the space $L^2(0, T; H_0^1(\Omega; w))$ admits an equivalent norm
\begin{equation*}
\|\varphi\|_{L^2(0, T; H^1(\Omega; w))}^2 = \iint_Q |\nabla \varphi|^2 w \, dx \, dt.
\end{equation*}
In the subsequent discussion, we will utilize this norm for $L^2(0, T; H^1(\Omega; w))$.
\end{remark}
Define
\begin{equation*}
W = \left\{\varphi \in L^2(0, T; H_0^1(\Omega; w)) \colon \partial_t \varphi \in L^2(0, T; H^{-1}(\Omega; w))\right\}.
\end{equation*}
It is clear that $W \subset C([0, T]; L^2(\Omega))$. Moreover, $W \hookrightarrow L^2(Q)$ is compact.
\begin{definition}\label{08.16.D1}
Let $f \in L^2(Q; w)$ and $\varphi_0 \in L^2(\Omega)$. We call $\varphi \in W$ a \emph{weak solution} of
\begin{equation}\label{03.01.1}
\begin{cases}
\partial_t \varphi - \mathrm{Div}(|x|^\alpha \nabla \varphi) = f, &\text{ in } Q, \\
\varphi = 0, &\text{ on } \partial Q, \\
\varphi(0) = \varphi_0, &\text{ in } \Omega,
\end{cases}
\end{equation}
if
\begin{equation*}
-\iint_Q \varphi \partial_t \psi \, dx \, dt + \iint_Q (\nabla \varphi \cdot \nabla \psi)w \, dx \, dt = \iint_Q f \psi \, dx \, dt + \int_\Omega \varphi_0(x)\psi(x, 0) \, dx
\end{equation*}
for any $\psi \in W$ with $\psi(T) = 0$.
\end{definition}
The following Lemma \ref{12.09.L1} is Lemma 3.7 in \cite{FF}.
\begin{lemma}\label{12.09.L1}
Let $\varphi_0 \in L^2(\Omega)$, $S = \sum_{i=1}^N \frac{\partial f_i}{\partial x_i} \in L^2(0, T; H^{-1}(\Omega; w))$, $g \in L^2(Q; w^{-1})$. If $\varphi \in L^2(0, T; H_0^1(\Omega; w))$ is a weak solution of the problem
\begin{equation*}
\begin{cases}
\partial_t \varphi - \mathrm{Div}(w \nabla \varphi) = g - S, &\text{ in } Q, \\
\varphi = 0, &\text{ on } \partial Q, \\
\varphi(0) = \varphi_0, &\text{ in } \Omega,
\end{cases}
\end{equation*}
then
\begin{equation}\label{12.09.1}
\begin{split}
&\sup_{t \in [0, T]} \int_\Omega |\varphi(x, t)|^2 \, dx + \iint_Q |\nabla \varphi(x, t)|^2 w \, dx \, dt \\
&\leq C\left(\|\varphi_0\|_{L^2(\Omega)}^2 + \|g\|_{L^2(Q; w^{-1})}^2 + \sum_{j=1}^N \|f_j\|_{L^2(\Omega; w^{-1})}^2\right),
\end{split}
\end{equation}
where the constant $C > 0$ depends only on $\alpha$, $N$, and $\Omega$.
\end{lemma}
 \begin{proof}
The proof is straightforward. Actually, we can readily verify that
\begin{equation*}
\begin{split}
&\iint_{\Omega \times (0, \tau)} w \nabla \varphi \cdot \nabla \varphi \, dx \, dt + \frac{1}{2} \int_\Omega |\varphi(x, \tau)|^2 \, dx \\
&= \frac{1}{2} \int_\Omega |\varphi_0|^2 \, dx + \iint_{\Omega \times (0, \tau)} g \varphi \, dx \, dt + \iint_{\Omega \times (0, \tau)} \sum_{i=1}^N \frac{\partial \varphi}{\partial x_i} f_i \, dx \, dt \\
&\leq \frac{1}{2} \int_\Omega |\varphi_0|^2 \, dx + \frac{1}{2\varepsilon} \iint_{\Omega \times (0, \tau)} g^2 w^{-1} \, dx \, dt + \varepsilon \iint_{\Omega \times (0, \tau)} \varphi^2 w \, dx \, dt \\
&  + \frac{1}{2\varepsilon} \iint_{\Omega \times (0, \tau)} \sum_{i=1}^N f_i^2 w^{-1} \, dx \, dt + \varepsilon \iint_{\Omega \times (0, \tau)} |\nabla \varphi|^2 w \, dx \, dt
\end{split}
\end{equation*}
holds for all $\tau \in (0, T)$ and $\varepsilon > 0$.
By choosing $\varepsilon = \frac{1}{4} \min\left\{\left(\frac{N - 2 + \alpha}{2m}\right)^2, 1\right\}$ and using \eqref{12.09.2}, we derive \eqref{12.09.1}.
\end{proof}

From Lemma \ref{12.09.L1}, we obtain the Corollary \ref{12.11.C1} following.
 \begin{corollary}\label{12.11.C1}
Under the conditions specified in Lemma \ref{12.09.L1}, where $g = f_i = 0$ for $i = 1, \cdots, N$, the following inequality holds:
\begin{equation*}
\int_\Omega |\varphi(x, \tau)|^2 \, dx \leq \int_\Omega |\varphi(x, t)|^2 \, dx
\end{equation*}
for all $0 \leq t \leq \tau \leq T$.
\end{corollary}

\section{Approximations}\label{Section 3}

 In this section, Theorem \ref{08.16.T1} stands as one of the principal results in this paper, while Corollary \ref{08.23.C1} is essential for Proposition \ref{12.23.P1}. Lemma \ref{12.17.L1} offers an approach to approximation (see, e.g., \cite{Wu1}), and Theorem \ref{08.16.T1} is founded upon this lemma.

 \begin{lemma}\label{12.17.L1}
There exists a $C^{2,1}$ function $\psi_\e: \R \to \R$ with $\e \in (0,R)$ such that
\begin{equation*}
\psi_\e(x) = |x| \text{  for } |x| \geq \e,   |x| \leq \psi_\e(x) \leq 2\e \text{  on } [-\e, \e],
\end{equation*}
and
\begin{equation}\label{12.16.1}
\psi_\e \geq \frac{\e}{4} \text{  on } \R,   |\psi_\e'| \leq C \text{  on } \R,   |\psi_\e''| \leq \frac{C}{\e}, \text{  and } |\psi_\e'''| \leq \frac{C}{\e^2} \text{  on } \R,
\end{equation}
where the constants $C > 0$ are absolute.
\end{lemma}
\begin{proof}
Consider the polynomial function
\begin{equation*}
\psi_\e(x) = \sum_{i=0}^4 a_i x^{4-i},
\end{equation*}
where $a_0, \ldots, a_4$ are constants to be determined. To this end, we impose the conditions
\begin{equation*}
\begin{split}
\psi_\e(-\e) = \e,   \psi_\e(\e) = \e,   \psi_\e'(-\e) = -1,   \psi_\e'(\e) = 1,   \psi_\e''(\pm \e) = 0,
\end{split}
\end{equation*}
which yield the solution
\begin{equation}\label{12.20.1}
\psi_\e(x) = \frac{3\e}{8} + \frac{3}{4\e}x^2 - \frac{1}{8\e^3}x^4,   \forall x \in [-\e,\e],   \text{ and}   \psi_\e(x) = |x|,   \forall |x| \geq \e.
\end{equation}
This is the desired function.
\end{proof}
\begin{remark}
Let $x \in \R^N$, and consider the function $\psi_\e(x) = \psi_\e(|x|)$ defined in Lemma \ref{12.17.L1}. We have
\begin{equation}\label{12.18.1}
\begin{split}
\nabla\psi_\e &= \frac{3}{2\e}x - \frac{1}{2\e^3}|x|^2x, \\
\frac{\pt^2 \psi_\e}{\pt x_i \pt x_j} &= \de_{ij}\left[\frac{3}{2\e} - \frac{1}{2\e^3}|x|^2\right] - \frac{1}{\e^3}x_ix_j, \\
\De\psi_\e &= N\left[\frac{3}{2\e} - \frac{1}{2\e^3}|x|^2\right] - \frac{1}{\e^3}|x|^2, \\
D^2\psi_\e\nabla\psi_\e &= \left(\frac{3}{2\e} - \frac{1}{2\e^3}|x|^2\right)^2x - \frac{1}{\e^3}\left(\frac{3}{2\e} - \frac{1}{2\e^3}|x|^2\right)|x|^2x,
\end{split}
\end{equation}
on $B_\e$, and
\begin{equation}\label{12.18.2}
\nabla \psi_\e = \frac{x}{|x|}, \;
\frac{\pt^2\psi_\e}{\pt x_i \pt x_j} = \de_{ij}|x|^{-1} - x_ix_j|x|^{-3},\;
\De\psi_\e = (N-1)|x|^{-1},
\end{equation}
on $\Om \setminus B_\e$.
\end{remark}
Denote
\begin{equation*}
\begin{split}
w_\e(x) = (\psi_\e(x))^\al = (\psi_\e(|x|))^\al,   x \in \Om.
\end{split}
\end{equation*}
\begin{lemma}\label{08.16.L2}
Let $z \in H_0^1(\Om; w_\e)$. Then,
\begin{equation}\label{08.20.1}
(N + \al - 2)\|w_\e^{\frac{1}{2} - \frac{1}{\al}}z\|_{L^2(\Om)} \leq 2\|z\|_{H_0^1(\Om; w_\e)}.
\end{equation}
Furthermore,
\begin{equation}\label{08.20.2}
\|z\|{L^2(\Om; w_\e)} \leq \frac{2m}{N + \al - 2}\|\nabla z\|_{L^2(\Om; w_\e)},
\end{equation}
and
\begin{equation}\label{08.31.10}
\|z\|{L^2(\Om)} \leq \frac{2m^{1 - \frac{\al}{2}}}{N + \al - 2}\|\nabla z\|_{L^2(\Om; w_\e)}.
\end{equation}
\end{lemma}
 \begin{proof}
As $w_\e \in C^{2,1}(\overline{\Omega})$, we derive
\begin{equation*}
\begin{split}
2\int_{\Omega} w_\e^{1 - \frac{2}{\alpha}}z(x \cdot \nabla z) \, dx
&= \int_{\Omega} w_\e^{1 - \frac{2}{\alpha}}x \cdot \nabla z^2 \, dx \\
&= \int_{\Omega} \Div\left(w_\e^{1 - \frac{2}{\alpha}}z^2x\right) \, dx - \int_{\Omega} z^2\Div\left(w_\e^{1 - \frac{2}{\alpha}}x\right) \, dx \\
&= -(N + \alpha - 2)\int_{\Omega} z^2w_\e^{1 - \frac{2}{\alpha}} \, dx - (2 - \alpha)\int_{B_\e}z^2\psi_\e^{\alpha - 3}[\psi_\e - x \cdot \nabla\psi_\e] \, dx.
\end{split}
\end{equation*}
Consequently,
\begin{equation*}
\begin{split}
(N + \alpha - 2)\int_{\Omega} z^2w_\e^{1 - \frac{2}{\alpha}} \, dx
&\leq -2\int_{\Omega} w_\e^{1 - \frac{2}{\alpha}}z(x \cdot \nabla z) \, dx \\
&= 2\int_{\Omega} \left(w_\e^{\frac{1}{2} - \frac{1}{\alpha}}z\right)\left(w_\e^{\frac{1}{2} - \frac{1}{\alpha}}x \cdot \nabla z\right) \, dx \\
&\leq 2\left(\int_{\Omega} w_\e^{1 - \frac{2}{\alpha}}z^2 \, dx\right)^{\frac{1}{2}}\left(\int_{\Omega} w_\e^{1 - \frac{2}{\alpha}}|x|^2|\nabla z|^2 \, dx\right)^{\frac{1}{2}} \\
&\leq 2\left(\int_{\Omega} w_\e^{1 - \frac{2}{\alpha}}z^2 \, dx\right)^{\frac{1}{2}}\left(\int_{\Omega} |\nabla z|^2w_\e \, dx\right)^{\frac{1}{2}},
\end{split}
\end{equation*}
where $\alpha \in (0,2)$ and $\psi_\e - x \cdot \nabla\psi_\e = \frac{3}{8\e^3}(\e^2 - |x|^2)^2 \geq 0$ on $B_\e$. This leads to
\begin{equation*}
(N + \alpha - 2)\|w_\e^{\frac{1}{2} - \frac{1}{\alpha}}z\|_{L^2(\Omega)} \leq 2\|\nabla z\|_{L^2(\Omega; w_\e)}.
\end{equation*}
Hence, \eqref{08.20.1} holds.
Finally, by \eqref{08.20.1} and $\e^\alpha \leq w_\e \leq m^\alpha$ in $\Omega$, we obtain \eqref{08.20.2} and \eqref{08.31.10}.
\end{proof}

We introduce some \textbf{Notations.} For $k \in \N$, we denote $w_k = w_\e$ with $\e = \frac{1}{k}$.
Fix $k \in \N$. Consider the equation
\begin{equation}\label{08.16.1}
\begin{cases}
\pt_t\wh \vp_k - \Div(w_k\nabla\wh \vp_k) = f_k, &\text{ in } Q, \\
\wh \vp_k = 0, &\text{ on } \pt Q, \\
\wh \vp_k(0) = \vp_k, & \text{ in } \Om,
\end{cases}
\end{equation}
where $\vp_k \in L^2(\Om)$, and $f_k$ is a given function satisfying $f_kw_k^{-1} \in L^2(Q; w_k)$ (i.e., $f_k \in L^2(Q; w_k^{-1})$).
Denote
\begin{equation*}
W_k = \left\{\vp \in L^2(0,T; H_{0}^1(\Om; w_k)) \colon \pt_t\vp \in L^2(0,T; H^{-1}(\Om; w_k))\right\}.
\end{equation*}
Clearly, $W_k \subset C([0,T]; L^2(\Om))$ by Lemma \ref{08.15.L1}. Furthermore, $W_k \hookrightarrow L^2(Q)$ is compact.
\begin{definition}\label{08.16.D1}
We call $\vp_k \in W_k$ a solution of \eqref{08.16.1} if
\begin{equation*}
-\iint_Q\vp_k\pt_t\psi \, dx \, dt + \iint_Q\left(\nabla \vp_k \cdot \nabla \psi\right)w_k \, dx \, dt = \iint_Q f_k\psi \, dx \, dt + \int_\Om \vp_k(x)\psi(x,0) \, dx
\end{equation*}
for any $\psi \in W_k$ with $\psi(T) = 0$.
\end{definition}

  \begin{lemma}\label{08.22.L1}
Let $\varphi_k \in W_k$ ($k \in \mathbb{N}$) denote the solution of \eqref{08.16.1}, where $\varphi_0 \in L^2(\Omega)$ and $f_k w_k^{-1} \in L^2(\Omega; w_k)$. Then, the following estimate holds:
\begin{equation*}
\max_{t \in [0,T]} \|\varphi_k(t)\|_{L^2(\Omega)} + \|\varphi_k\|_{L^2(0,T; H_0^1(\Omega; w_k))} \leq C \left( \|\varphi_0\|_{L^2(\Omega)} + \|f_k w_k^{-1}\|_{L^2(\Omega; w_k)} \right),
\end{equation*}
where the positive constant $C$ depends solely on $\alpha$, $N$, and $\Omega$.
Moreover, if $f_k = 0$, then for all $0 \leq t \leq \tau \leq T$, the following inequality is satisfied:
\begin{equation*}
\|\varphi_k(\tau)\|_{L^2(\Omega)} \leq \|\varphi_k(t)\|_{L^2(\Omega)}.
\end{equation*}
\end{lemma}
 \begin{proof}
Let $\varphi_k \in W_k$ denote the test function. By multiplying $\varphi_k$ to both sides of \eqref{08.16.1} and integrating over the interval $(0,t)$, we derive
\begin{equation*}
\begin{split}
&\frac{1}{2}\int_\Omega \varphi_k^2(t)\mathrm{d} x + \iint_{\Omega\times (0,t)} |\nabla\varphi_k|^2w_k\mathrm{d} x\mathrm{d} t \\
&= \iint_{\Omega\times (0,t)} f_k\varphi_k\mathrm{d} x + \frac{1}{2}\int_\Omega \varphi_0^2\mathrm{d} x \\
&\leq \delta \iint_{\Omega\times (0,t)} \varphi_k^2w_k\mathrm{d} x\mathrm{d} t + \frac{1}{2\delta}\iint_{\Omega\times (0,1)} f_k^2w_k\mathrm{d} x\mathrm{d} t + \frac{1}{2}\int_\Omega \varphi_0^2\mathrm{d} x.
\end{split}
\end{equation*}
By applying \eqref{08.20.2} from Lemma \ref{08.16.L2} and setting $\delta = \frac{N+\alpha-2}{4m}$, we obtain
\begin{equation*}
\int_\Omega \varphi_k^2(t)\mathrm{d} x + \iint_{\Omega\times (0,t)} |\nabla\varphi_k|^2w_k\mathrm{d} x\mathrm{d} t \leq \frac{2m}{N+\alpha-2}\iint_{\Omega\times (0,t)} f_k^2w_k\mathrm{d} x\mathrm{d} t + \int_\Omega \varphi_0^2\mathrm{d} x.
\end{equation*}
Defining $C = \max\left\{\frac{2m}{N+\alpha-2}, 1\right\}$, we conclude the proof of Lemma \ref{08.22.L1}.
\end{proof}
 \begin{theorem}\label{08.16.T1}
Let $\widehat{\varphi}_0 \in W$ denote the solution of equation \eqref{12.14.1} with initial data $\varphi_0 \in L^2(\Omega)$ and $fw^{-1} \in L^2(\Omega;w)$. For $k \in \mathbb{N}$, let $\widehat{\varphi}_k$ be the solution of equation \eqref{08.16.1} with initial data $\varphi_k = \varphi_0$ and source term $f_k = f$. Then, the following weak convergence results hold:
\begin{equation}\label{08.22.1}
\begin{split}
\widehat{\varphi}_k &\rightharpoonup \widehat{\varphi}_0   \text{  weakly in } L^2(0,T; H_0^1(\Omega;w)), \\
\widehat{\varphi}_k &\rightharpoonup \widehat{\varphi}_0   \text{  weakly in } L^2(Q).
\end{split}
\end{equation}
Furthermore, we have:
\begin{equation}\label{12.09.5}
\widehat{\varphi}_k(T) \rightharpoonup \widehat{\varphi}_0(T) \text{  weakly in } L^2(\Omega).
\end{equation}
\end{theorem}
\begin{proof}  The proof will be divided into several steps.

{\it Step 1.}  Observe that
\begin{equation}\label{12.11.1}
\begin{split}
\iint_Q (fw_k^{-1})^2w_k\,\mathrm{d} x\,\mathrm{d} t &= \iint_{Q}(fw^{-1})^2(ww_k^{-1})w\,\mathrm{d} x\,\mathrm{d} t \\
&\leq \iint_Q (fw^{-1})^2w\,\mathrm{d} x\,\mathrm{d} t,
\end{split}
\end{equation}
and
\begin{equation*}
\iint_{\Omega\times (0,t)}|\nabla\widehat{\varphi}_k|^2w\,\mathrm{d} x\,\mathrm{d} t \leq \iint_{\Omega\times (0,t)}|\nabla \widehat{\varphi}_k|^2w_k\,\mathrm{d} x\,\mathrm{d} t,
\end{equation*}
since $w\leq w_k$ in $\Omega$. By Lemma \ref{08.22.L1}, we have
\begin{equation*}
\begin{split}
\max_{t\in [0,T]}\|\widehat{\varphi}_k(t)\|_{L^2(\Omega)}+\|\widehat{\varphi}_k\|_{L^2(0,T; H_0^1(\Omega;w))} \leq C\left(\|\varphi_0\|_{L^2(\Omega)}+\|fw^{-1}\|_{L^2(\Omega; w)}\right),
\end{split}
\end{equation*}
where the constant $C>0$ depends only on $\alpha$, $N$, and $\Omega$. Consequently, there exists a subsequence of $\{\widehat{\varphi}k\}{k\in\mathbb{N}}$, still denoted by itself, and $\widetilde{\varphi}_0\in L^2(0,T; H_0^1(\Omega;w))$, such that
\begin{equation}\label{08.16.4}
\begin{split}
\widehat{\varphi}_k
&\rightharpoonup \widetilde{\varphi}_0   \text{  weakly in } L^2(0,T; H_0^1(\Omega;w)),
\end{split}
\end{equation}
and by Lemma \ref{08.15.L1}, we obtain
\begin{equation}\label{08.23.1}
\widehat{\varphi}_k
\rightharpoonup \widetilde{\varphi}_0   \text{  weakly in } L^2(Q).
\end{equation}
Furthermore, we have
\begin{equation*}
\widehat{\varphi}_k
\rightharpoonup \widetilde{\varphi}_0   \text{  weakly in } L^2(Q; w).
\end{equation*}

 {\it Step 2}. Next, we aim to prove that $\widetilde{\varphi}_0$ constitutes a weak solution of \eqref{12.14.1}.
Assume $\psi\in W$ with $\psi(T)=0$. By leveraging a density argument, we can assume $\psi\in C^\infty(\overline{Q})$ and that $\psi(\cdot, t)$ is compactly supported in $\Omega$ for any $t$. Utilizing the definitions of weak solution (refer to Definitions \ref{08.16.D1} and \ref{08.16.D1}) and equations \eqref{08.16.4} and \eqref{08.23.1}, it suffices to demonstrate
\begin{equation}\label{08.16.5}
\iint_Q(\nabla\widehat{\varphi}_k\cdot \nabla\psi)w_k\mathrm{d} x\mathrm{d} t\to \iint_Q(\nabla\widetilde{\varphi}_0\cdot\nabla\psi)w\mathrm{d} x\mathrm{d} t   \text{  as } k\to \infty.
\end{equation}
Firstly, for any $\gamma>0$, by \eqref{08.16.4}, there exists $k_0\in\mathbb{N}$ such that for all $k\geq k_0$, we obtain
\begin{equation}\label{08.16.6}
\left|\iint_Q(\nabla\widehat{\varphi}_k\cdot\nabla\psi)w\mathrm{d} x\mathrm{d} t-\iint_Q(\nabla\widetilde{\varphi}_0\cdot\nabla\psi)w\mathrm{d} x\mathrm{d} t\right|<\frac{1}{2}\gamma.
\end{equation}
Secondly, observe that $w=w_k$ on $\Omega\setminus B_{\frac{1}{k}}$, implying
\begin{equation}\label{08.16.7}
\iint_{(\Omega\setminus B_{\frac{1}{k}})\times (0,T)}(\nabla\widehat{\varphi}_k\cdot\nabla\psi)w_k\mathrm{d} x\mathrm{d} t=\iint_{(\Omega\setminus B_{\frac{1}{k}})\times (0,T)}(\nabla \widehat{\varphi}_k\cdot\nabla\psi)w\mathrm{d} x\mathrm{d} t.
\end{equation}
 Thirdly, by utilizing the same reasoning as employed in \eqref{12.11.1} from Step 1, we obtain
\begin{eqnarray}\label{08.16.9}
&&\left|\iint_{B_{\frac{1}{k}} \times (0,T)} (\nabla\widehat{\varphi}_k \cdot \nabla\psi) w_k \, dx \, dt\right| \crr
&&\leq \left(\iint_{B_{\frac{1}{k}} \times (0,T)} |\nabla\widehat{\varphi}_k|^2 w_k \, dx \, dt\right)^{\frac{1}{2}} \left(\iint_{B_{\frac{1}{k}} \times (0,T)} |\nabla\psi|^2 w_k \, dx \, dt\right)^{\frac{1}{2}} \crr
&&\leq C\sqrt{T} \left(\sup_{Q} |\nabla\psi|\right) \left(\|\varphi_0\|_{L^2(\Omega)} + \|f w^{-1}\|_{L^2(Q; w)}\right) \left(w_k(B_{\frac{1}{k}})\right)^{\frac{1}{2}} \crr
&&\leq C_\psi \sqrt{T} \left(\|\varphi_0\|_{L^2(\Omega)} + \|f w^{-1}\|_{L^2(\Omega; w)}\right) \frac{1}{k^{\frac{N+\alpha}{2}}}
\end{eqnarray}
given that
\begin{equation*}
w_k(B_{\frac{1}{k}}) = \int_{B_{\frac{1}{k}}} w_k \, dx = \int_{B_{\frac{1}{k}}} \frac{1}{k^\alpha} \, dx = C \frac{1}{k^{N+\alpha}},
\end{equation*}
where the constant $C_\psi$ depends exclusively on $\alpha, N, \Omega,$ and $\psi$.
Fourthly,
\begin{equation}\label{08.16.10}
\begin{split}
&\left|\iint_{B_{\frac{1}{k}} \times (0,T)} (\nabla\widehat{\varphi}_k \cdot \nabla\psi) w \, dx \, dt\right| \\
&\leq \left(\iint_{B_{\frac{1}{k}} \times (0,T)} |\nabla\widehat{\varphi}_k|^2 w \, dx \, dt\right)^{\frac{1}{2}} \left(\iint_{B_{\frac{1}{k}} \times (0,T)} |\nabla\psi|^2 w \, dx \, dt\right)^{\frac{1}{2}} \\
&\leq C\sqrt{T} \left(\sup_{Q} |\nabla\psi|\right) \left(\|\varphi_0\|_{L^2(\Omega)} + \|f\|_{L^2(Q; w)}\right) \left(w(B_{\frac{1}{k}})\right)^{\frac{1}{2}} \\
&\leq C_\psi \sqrt{T} \left(\|\varphi_0\|_{L^2(\Omega)} + \|f\|_{L^2(Q; w)}\right) \frac{1}{k^{\frac{N+\alpha}{2}}}
\end{split}
\end{equation}
in accordance with
\begin{equation*}
w(B_{\frac{1}{k}}) = \int_{B_{\frac{1}{k}}} w \, dx = \int_{B_{\frac{1}{k}}} |x|^\alpha \, dx = C \frac{1}{k^{N+\alpha}},
\end{equation*}
where the constant $C_\psi$ is contingent only on $\alpha, N, \Omega,$ and $\psi$.
Ultimately, since $w = w_k$ on $\Omega \setminus B_{\frac{1}{k}}$, by \eqref{08.16.6}, \eqref{08.16.7}, \eqref{08.16.9}, and \eqref{08.16.10}, we derive
\begin{eqnarray*}
&&\left|\iint_Q (\nabla\widehat{\varphi}_k \cdot \nabla\psi) w_k \, dx \, dt - \iint_Q (\nabla\widetilde{\varphi}_0 \cdot \nabla\psi) w \, dx \, dt\right| \\
&&\leq \left|\iint_Q (\nabla\widehat{\varphi}_k \cdot \nabla\psi) w_k \, dx \, dt - \iint_Q (\nabla\widehat{\varphi}_k \cdot \nabla\psi) w \, dx \, dt\right| \\
&&  + \left|\iint_Q (\nabla\widehat{\varphi}_k \cdot \nabla\psi) w \, dx \, dt - \iint_Q (\nabla\widetilde{\varphi}_0 \cdot \nabla\psi) w \, dx \, dt\right| \\
&&\leq \left|\iint_{(\Omega \setminus B_{\frac{1}{k}}) \times (0,T)} (\nabla\widehat{\varphi}_k \cdot \nabla\psi) w_k \, dx \, dt - \iint_{(\Omega \setminus B_{\frac{1}{k}}) \times (0,T)} (\nabla\widehat{\varphi}_k \cdot \nabla\psi) w \, dx \, dt\right| \\
&&  + \left|\iint_{B_{\frac{1}{k}} \times (0,T)} (\nabla\widehat{\varphi}_k \cdot \nabla\psi) w_k \, dx \, dt\right| + \left|\iint_{B_{\frac{1}{k}} \times (0,T)} (\nabla\widehat{\varphi}_k \cdot \nabla\psi) w \, dx \, dt\right| + \frac{1}{2}\gamma \\
&&\leq C_\psi \sqrt{T} \left(\|\varphi_0\|_{L^2(\Omega)} + \|f\|_{L^2(Q; w)}\right) \frac{1}{k^{\frac{N+\alpha}{2}}} + \frac{1}{2}\gamma,
\end{eqnarray*}
thus validating \eqref{08.16.5}. From this and {\it Step 1}, we have established that
\begin{equation*}
-\iint_Q \widetilde{\varphi}_0 \partial_t \psi \, dx \, dt + \iint_Q (\nabla\widetilde{\varphi}_0 \cdot \nabla\psi) w \, dx \, dt = \iint_Q f \psi \, dx \, dt + \int_\Omega \varphi_0(x) \psi(x,0) \, dx.
\end{equation*}
This confirms that $\widetilde{\varphi}_0$ is a solution of \eqref{12.14.1}.

{\it Step 4}. Given that $\widetilde{\varphi}_0$ and $\widehat{\varphi}_0$ are solutions of \eqref{12.14.1} with initial data $\varphi_0 \in L^2(\Omega)$ and $f w^{-1} \in L^2(\Omega; w)$, by the uniqueness of the solution of \eqref{12.14.1}, we conclude that $\widetilde{\varphi}_0 = \widehat{\varphi}_0$.

{\it Step 5}. Finally, let $\psi \in C^\infty(\overline{Q})$ and $\psi(\cdot, t)$ be compactly supported in $\Omega$ for any $t \in [0,T]$. Multiplying $\psi$ on both sides of \eqref{08.16.1} with $\widehat{\varphi}_k(0) = \varphi_0$ and $f_k = f$, we obtain
\begin{equation*}
\begin{split}
\int_\Omega \widehat{\varphi}_k(T) \psi(T) \, dx - \iint_Q \widehat{\varphi}_k \partial_t \psi \, dx \, dt + \iint_Q w_k \nabla \widehat{\varphi}_k \cdot \nabla\psi \, dx \, dt = \iint_Q f \psi \, dx \, dt + \int_\Omega \widehat{\varphi}_k(0) \psi(0) \, dx.
\end{split}
\end{equation*}
Multiplying $\psi$ on both sides of \eqref{12.14.1}, we obtain
\begin{equation*}
\int_\Omega \widehat{\varphi}_0(T) \psi(T) \, dx - \iint_Q \widehat{\varphi}_0 \partial_t \psi \, dx \, dt + \iint_Q w \nabla \widehat{\varphi}_0 \cdot \nabla\psi \, dx \, dt = \iint_Q f \psi \, dx \, dt + \int_\Omega \widehat{\varphi}_0(0) \psi(0) \, dx.
\end{equation*}
By \eqref{08.16.5} and $\widehat{\varphi}_0 = \widetilde{\varphi}_0$, $\widehat{\varphi}_k(0) = \widehat{\varphi}_0(0) = \varphi_0$ for all $k \in \mathbb{N}$, and $\widehat{\varphi}_k \rightharpoonup \widehat{\varphi}_0$ weakly in $L^2(Q)$, we obtain
\begin{equation*}
\int_\Omega \widehat{\varphi}_k(T) \psi(T) \, dx \to \int_\Omega \widehat{\varphi}_0(T) \psi(T) \, dx.
\end{equation*}
This indicates that $\widehat{\varphi}_k(T) \to \widehat{\varphi}_0(T)$ in $L^2(\Omega)$ since $\psi(T)$ is arbitrary.
\end{proof}
 \begin{corollary}\label{08.23.C1}
Under the hypotheses stated in Theorem \ref{08.16.T1}, and with the additional assumptions that the initial data $\varphi_0$ belongs to $H_0^1(\Omega) \cap H^2(\Omega)$ and $f_k = f$ for all $k \in \mathbb{N}$,  the following strong convergence results hold:
\begin{equation*}
\widehat{\varphi}_k \to \widehat{\varphi}_0   \quad \text{in } L^2(Q),
\end{equation*}
and
\begin{equation*}
\widehat{\varphi}_k(T) \to \widehat{\varphi}_0(T)   \quad \text{in } L^2(\Omega).
\end{equation*}
Furthermore, if $f = 0$, $\varphi_0 \in H_0^1(\Omega) \cap H^3(\Omega)$, and $\Div(w_k \nabla \varphi_0) \in H_0^1(\Omega)$, then the following additional strong convergence results hold:
\begin{equation*}
\frac{\partial \widehat{\varphi}_k}{\partial \nu} \to \frac{\partial \widehat{\varphi}_0}{\partial \nu}   \quad \text{in } L^2(\partial Q),
\end{equation*}
and
\begin{equation*}
\nabla \widehat{\varphi}_k \to \nabla \widehat{\varphi}_0   \quad \text{in } [L^2(K \times (0,T))]^N
\end{equation*}
for every measurable subset $K \subset \Omega$ satisfying $\dist(0, K) > 0$.
\end{corollary}

 \begin{proof}  The proof will be split into several steps.

 Given $\varphi_0 \in H_0^1(\Omega) \cap H^2(\Omega)$, the solution $\widehat{\varphi}_k$ of \eqref{08.16.1} satisfies (noting that $H_{w_k,0}^1 = H_0^1(\Omega)$ for each $k \in \mathbb{N}$)
\begin{equation}\label{12.25.1}
\widehat{\varphi}_k \in L^2(0,T; H^3(\Omega) \cap H_0^1(\Omega)) \cap H^1(0,T; H_0^1(\Omega))   \quad \text{and} \quad   \partial_t \widehat{\varphi}_k \in L^2(0,T; H_0^1(\Omega)).
\end{equation}
Multiplying $\partial_t \widehat{\varphi}_k$ on both sides of \eqref{08.16.1} and integrating over $\Omega \times (0,t)$, we get
\begin{equation*}
\begin{split}
\iint_{\Omega \times (0,t)} (\partial_t \widehat{\varphi}_k)^2 \, dx \, dt + \frac{1}{2} \iint_{\Omega \times (0,t)} \partial_t \left( |\nabla \widehat{\varphi}_k|^2 w_k \right) \, dx \, dt = \iint_{\Omega \times (0,t)} f \partial_t \widehat{\varphi}_k \, dx \, dt.
\end{split}
\end{equation*}
This leads to
\begin{equation*}
\begin{split}
&\iint_{\Omega \times (0,t)} |\partial_t \widehat{\varphi}_k|^2 \, dx \, dt + \frac{1}{2} \int_\Omega |\nabla \widehat{\varphi}_k(t)|^2 w_k \, dx \\
&= \iint_{\Omega \times (0,t)} f \partial_t \widehat{\varphi}_k \, dx \, dt + \frac{1}{2} \int_\Omega |\nabla \widehat{\varphi}_k(0)|^2 w_k \, dx \\
&\leq \frac{1}{2} \iint_{\Omega \times (0,t)} |\partial_t \widehat{\varphi}_k|^2 \, dx \, dt + \frac{m^\alpha}{2} \left( \iint_{\Omega \times (0,t)} (f w^{-1})^2 w \, dx \, dt + \int_\Omega |\nabla \varphi_0|^2 \, dx \right).
\end{split}
\end{equation*}
for all $t \in (0,T]$.
Consequently:
\begin{equation}\label{12.09.3}
\iint_{\Omega \times (0,t)} |\partial_t \widehat{\varphi}_k|^2 \, dx \, dt + \int_\Omega |\nabla \widehat{\varphi}_k(t)|^2 w_k \, dx \leq m^\alpha \left( \iint_{\Omega \times (0,t)} (f w^{-1})^2 w \, dx \, dt + \int_\Omega |\nabla \varphi_0|^2 \, dx \right).
\end{equation}
Thus, by \eqref{08.15.10}, we have
\begin{equation}\label{08.23.2}
\begin{split}
\|\partial_t \widehat{\varphi}_k\|_{L^2(0,T; H_w^{-1}(\Omega))}
&= \sup_{\|\psi\|_{L^2(0,T; H_{w,0}^1(\Omega))} \leq 1} \langle \partial_t \widehat{\varphi}_k, \psi \rangle_{L^2(0,T; H_w^{-1}(\Omega)), L^2(0,T; H_{w,0}^1(\Omega))} \\
&= \sup_{\|\psi\|_{L^2(0,T; H_{w,0}^1(\Omega))} \leq 1} \iint_{\Omega \times (0,t)} (\partial_t \widehat{\varphi}_k) \psi \, dx \, dt \\
&\leq C \sup_{\|\psi\|_{L^2(0,T; H_{w,0}^1(\Omega))} \leq 1} \|\partial_t \widehat{\varphi}_k\|_{L^2(\Omega \times (0,t))} \|\psi\|_{L^2(\Omega \times (0,t))} \leq C \|\partial_t \widehat{\varphi}_k\|_{L^2(\Omega \times (0,t))},
\end{split}
\end{equation}
where the constant $C > 0$ depends only on $\alpha, N$, and $\Omega$, and we utilized \eqref{08.15.10} in the last inequality. From this and Theorem \ref{08.16.T1}, we obtain that there exists a subsequence of $\{\widehat{\varphi}_k\}$, still denoted by itself, such that:
\begin{equation*}
\partial_t \widehat{\varphi}_k \rightharpoonup \partial_t \widehat{\varphi}_0     \text{ weakly in } L^2(0,T; H_w^{-1}(\Omega)).
\end{equation*}
From this, \eqref{08.23.2}, and the compact embedding $W \hookrightarrow L^2(\Omega \times (0,t))$, we have:
\begin{equation*}
\widehat{\varphi}_k \to \widehat{\varphi}_0     \text{ strongly in } L^2(\Omega \times (0,t)).
\end{equation*}

   {\it Step 2}. Next, utilizing \eqref{12.09.3} and the inequality $w \leq w_k$ for all $k \in \mathbb{N}$, we derive:
\begin{equation*}
\int_\Omega |\nabla \widehat{\varphi}_k(T)|^2 w \, dx \leq \int_\Omega |\nabla \widehat{\varphi}_k(T)|^2 w_k \, dx \leq m^\alpha \left( \iint_Q (f w^{-1})^2 w \, dx \, dt + \int_\Omega |\nabla \varphi_0|^2 \, dx \right).
\end{equation*}
This demonstrates the existence of a subsequence of $\{\widehat{\varphi}_k(T)\}$, still denoted by itself, and $\xi \in H_{w,0}^1(\Omega)$, such that $\widehat{\varphi}_k(T) \rightharpoonup \xi$ weakly in $H_{w,0}^1(\Omega)$, and $\widehat{\varphi}_k(T) \to \xi$ strongly in $L^2(\Omega)$ by Lemma \ref{08.15.L4}. This, coupled with \eqref{12.09.5}, yields $\xi = \widehat{\varphi}_0(T)$.
This concludes the proof of Corollary \ref{08.23.C1}.

{\it Step 3}. Finally, selecting $\zeta \in C^\infty(\mathbb{R}^N), 0 \leq \zeta \leq 1$ satisfying:
\begin{equation*}
\zeta = 1   \text{ on } \mathbb{R}^N \setminus B_{2R},   \zeta = 0   \text{ on } B_R,   |\nabla \zeta| \leq \frac{C}{R}   \text{ on } \mathbb{R}^N
\end{equation*}
with $\text{ dist}(K,0) \geq 2R > 0$ and $B_{4R} \subset \Omega$, we apply the trace theorem to obtain:
\begin{equation}\label{12.22.2}
\begin{split}
\iint_{\partial Q} \left( \frac{\partial \widehat{\varphi}_k}{\partial \nu} \right)^2 \, dS \, dt
&\leq C \sum_{k=0}^2 \iint_Q \left| D^k (\zeta \widehat{\varphi}_k) \right|^2 \, dx \, dt,
\end{split}
\end{equation}
where $C > 0$ is a constant dependent solely on $\Omega$. Observing that $h_k = \zeta \widehat{\varphi}_k$ solves the equation:
\begin{equation*}
\begin{cases}
\partial_t h_k - \Div(w_k \nabla h_k) = \zeta f - 2w_k \nabla \zeta \cdot \nabla \widehat{\varphi}_k - \widehat{\varphi}_k \Div(w_k \nabla \zeta), &\text{ in } (\Omega \setminus B_R) \times (0,T), \\
h_k = 0, &\text{ on } \partial (\Omega \setminus B_R) \times (0,T), \\
h_k(0) = \zeta \varphi_0, &\text{ in } \Omega \setminus B_R,
\end{cases}
\end{equation*}
which is uniformly parabolic. Given $f = 0$, $\zeta \varphi_0 \in H_0^1(\Omega \setminus B_R)$, $\Div(w_k (\nabla (\zeta \varphi_0)) \in H_0^1(\Omega \setminus B_R)$, and $\partial_t [2w_k \nabla \zeta \cdot \nabla \widehat{\varphi}_k + \widehat{\varphi}_k \Div(w_k \nabla \zeta)] \in L^2(Q)$ by \eqref{12.25.1}, we invoke the standard regularity theory for uniformly parabolic equations (see Theorem 5 in Chapter 7.1.3 in \cite{Evans}, pages 360-361) to establish the existence of a constant $C > 0$ dependent on $\alpha, T, R$, and $\Omega$, such that:
\begin{equation}\label{12.25.2}
\begin{split}
\|\partial_t h_k\|_{L^2(0,T; H^2(\Omega \setminus B_R))} + \|h_k\|_{L^2(0,T; H^3(\Omega \setminus B_R))}
\leq C \|\varphi_0\|_{H^3(\Omega)}.
\end{split}
\end{equation}
This implies that $\{\frac{\partial \widehat{\varphi}_k}{\partial \nu}\}$ is bounded in $L^2(0,T; H^1(\partial \Omega))$. By the classical Sobolev compact embedding theorem, there exists a subsequence of $\{\frac{\partial \widehat{\varphi}_k}{\partial \nu}\}$, still denoted by itself, such that:
\begin{equation*}
\frac{\partial \widehat{\varphi}_k}{\partial \nu} \to \frac{\partial \widehat{\varphi}_0}{\partial \nu}   \text{ strongly in } L^2(\partial Q).
\end{equation*}
Analogously to \eqref{08.23.2}, noting that $H_{w,0}^1(\Omega \setminus B_R) = H_0^1(\Omega \setminus B_R)$ from $R^\alpha \leq w \leq m^\alpha$ on $\Omega \setminus B_R$, we utilize \eqref{12.09.3} to obtain:
\begin{equation*}
\begin{split}
\|\partial_t \widehat{\varphi}_k\|_{L^2(0,T; L^2(\Omega \setminus B_R; w))}
&= \sup_{\|\psi\|_{L^2(0,T; L^2(\Omega \setminus B_R; w))} \leq 1} \iint_Q (\partial_t \widehat{\varphi}_k) \psi w \, dx \, dt \\
&\leq C \|\partial_t \widehat{\varphi}_k\|_{L^2(0,T; L^2(\Omega \setminus B_R; w))} \|\psi\|_{L^2(0,T; L^2(\Omega \setminus B_R; w))} \leq C \|\varphi_0\|_{H^2(\Omega \setminus B_R)}.
\end{split}
\end{equation*}
By \eqref{12.09.3}, \eqref{12.25.2}, $L^2(\Om\se B_R;w)=L^2(\Om\se B_R)$, $\partial_t h_k = \zeta \partial_t \widehat{\varphi}_k$ for each $k \in \mathbb{N}$, and the compact embedding:
\begin{equation*}
\{h \in L^2(0,T; H_0^1(\Omega \setminus B_R) \cap H^2(\Omega \setminus B_R))\colon \partial_t h \in L^2(0,T; L^2(\Omega \setminus B_R))\} \hra L^2(0,T; H_0^1(\Omega \setminus B_R)),
\end{equation*}
we deduce:
\begin{equation*}
\nabla h_k \to \nabla (\zeta \widehat{\varphi}_0)   \text{ strong in } [L^2(0,T; L^2(\Omega \setminus B_R))]^N.
\end{equation*}
Furthermore, given $\zeta = 0$ on $\Omega \setminus B_{2R}$, we conclude:
\begin{equation*}
\nabla \widehat{\varphi}_k \to \nabla \widehat{\varphi}_0   \text{ strong in } [L^2(K \times (0,T))]^N.
\end{equation*}
This completes the proof of the corollary.
\end{proof}

\section{Carleman estimates}\label{Section 4}

   In this section, we start  by conducting a series of calculations aimed at deriving Theorem \ref{12.25.T1}. It is worth noting that the computation of the Carleman estimate hinges on the selection of weight functions (as exemplified by the functions $\eta_\varepsilon\ (\varepsilon>0)$ presented below). Consequently, the Carleman estimate stated in Theorem \ref{12.21.T1} is contingent upon the parameter $\varepsilon>0$. Upon examining Theorem \ref{12.21.T1}, we observe that this Carleman estimate deviates from the standard form due to the presence of an additional term (refer to the last term in Theorem \ref{12.21.T1}). To address this, we partition the estimation of the Carleman estimate into three distinct parts, which constitutes a novel contribution of this paper. The approximate theorem \ref{12.23.T1} serves as a cornerstone for the Carleman estimate of the degenerate parabolic equation \eqref{12.14.1}. The derivation of Theorem \ref{12.25.T1} from Theorem \ref{12.23.T1} within the framework of Carleman estimates follows a standard procedure. Subsequently, we turn our attention to the following backward equation:
\begin{equation}\label{12.15.1}
\begin{cases}
\partial_t u + \Div(w_\varepsilon \nabla u) = g, &\text{ in } Q, \\
u = 0, &\text{ on } \partial Q,\\
u(T) = u_T, &\text{ in } \Omega.
\end{cases}
\end{equation}
where $u_T \in H_0^1(\Omega) \cap H^3(\Omega)$, $g \in L^2(Q; w_\varepsilon^{-1})$.
Assume $0 < \varepsilon \ll R$.
We assume $\eta_\varepsilon = \gamma(-2m^{2-\alpha} + \psi_\varepsilon^{2-\alpha})$, where the constant $\gamma > 0$ will be specified later. Note that $-2m^{2-\alpha} \leq \eta_\varepsilon \leq -m^{2-\alpha}$ on $\Omega$.
Taking
\begin{equation*}
\Theta(t) = [t(T-t)]^{-4},   \quad \xi(x,t) = \Theta(t)\eta_\varepsilon(x),
\end{equation*}
and
\begin{equation*}
u = e^{-s\xi}v \Longleftrightarrow v = e^{s\xi}u.
\end{equation*}
Then, the following properties hold for $v$:

i) $v = \frac{\partial v}{\partial x_i} = 0$ in $L^2(\Omega; w_\varepsilon)$ at $t=0$ and $t=T$;

ii) $v = 0$ on $\partial Q$.

Set
\begin{equation*}
\begin{split}
P_1v
&= \sum_{i=1}^3 P_{1i}v := v_t - 2sw_\varepsilon\nabla v \cdot \nabla \xi - sv\Div(w_\varepsilon \nabla\xi),\\
P_2v
&= \sum_{i=1}^3 P_{2i}v := \Div(w_\varepsilon\nabla v) - s\xi_t v + s^2vw_\varepsilon \nabla\xi \cdot \nabla\xi.
\end{split}
\end{equation*}
Then,
\begin{equation*}
e^{s\xi}g = P_1v + P_2v.
\end{equation*}
We have
\begin{equation*}
\begin{split}
\nabla w_\varepsilon &= \alpha\psi_\varepsilon^{\alpha-1}\nabla\psi_\varepsilon,  \quad \nabla \eta_\varepsilon = \gamma(2-\alpha)\psi_\varepsilon^{1-\alpha}\nabla\psi_\varepsilon,  \quad \nabla \xi = \gamma (2-\alpha)\Theta \psi_\varepsilon^{1-\alpha}\nabla\psi_\varepsilon.
\end{split}
\end{equation*}
and
\begin{equation}\label{12.21.2}
\begin{split}
\Div(w_\varepsilon\nabla\xi) &= \gamma(2-\alpha)\Theta (|\nabla\psi_\varepsilon|^2 + \psi_\varepsilon \Delta\psi_\varepsilon),   \quad \xi_t = \Theta'\eta_\varepsilon,   \quad \xi_{tt} = \Theta''\eta_\varepsilon.
\end{split}
\end{equation}
Now, we compute $(P_1v, P_2v)_{L^2(Q)}$ term-by-term.

{\it (i):  Compute $(P_{11}v, P_{21}v)_{L^2(Q)}$.}

Indeed, by $v_t = 0$ on $\partial Q$ and i), we have
\begin{equation*}
\begin{split}
(P_{11}v, P_{21}v)_{L^2(Q)}
&= +\iint_Q v_t\Div(w_\varepsilon\nabla v)\df x\df t = -\iint_Q w_\varepsilon \nabla v \cdot \nabla v_t\df x\df t\\
&= -\frac{1}{2}\iint_Q\partial_t (w_\varepsilon\nabla v \cdot\nabla v)\df x\df t = 0.
\end{split}
\end{equation*}

{\it (ii): Compute $(P_{12}v, P_{21}v)_{L^2(Q)}$.}

Indeed, by $v = 0$ on $\partial Q$, we have
\begin{eqnarray*}
&&(P_{12}v,P_{21}v)_{L^2(Q)}\\
&&= -2s\iint_Q (w_\varepsilon \nabla v \cdot \nabla \xi )\Div(w_\varepsilon \nabla v)\df x\df t\\
&&= -2s\iint_Q\Div\left[(w_\varepsilon \nabla v \cdot\nabla \xi)(w_\varepsilon\nabla v)\right]\df x\df t + 2s\iint_Q w_\varepsilon \nabla v \cdot \nabla (w_\varepsilon \nabla v \cdot \nabla\xi)\df x\df t\\
&&= -2\gamma(2-\alpha)s\iint_{\partial Q} \Theta \psi_\varepsilon^\alpha\left(\frac{\partial v}{\partial \nu}\right)^2(x\cdot\nu)\df S\df t\\
&&\hspace{4.5mm}+2s\iint_Q (w_\varepsilon\nabla v \cdot \nabla w_\varepsilon)(\nabla v\cdot\nabla\xi)\df x\df t + 2s\iint_Q w_\varepsilon^2\nabla v \cdot\nabla (\nabla v\cdot \nabla \xi)\df x\df t\\
&&= -\gamma(2-\alpha)s\iint_{\partial Q}\Theta \psi_\varepsilon^\alpha \left(\frac{\partial v}{\partial \nu}\right)^2(x\cdot \nu)\df S\df t - \gamma(2-\alpha)(1+\alpha)s\iint_Q\Theta \psi_\varepsilon^\alpha |\nabla v|^2|\nabla\psi_\varepsilon|^2\df x\df t\\
&&\hspace{4.5mm}+2\gamma(2-\alpha)s\iint_Q\Theta \psi_\varepsilon^\alpha (\nabla v\cdot\nabla \psi_\varepsilon)^2\df x\df t\\
&&\hspace{4.5mm}-\gamma(2-\alpha)s\iint_Q\Theta \psi_\varepsilon^\alpha|\nabla v|^2\psi_\varepsilon \Delta\psi_\varepsilon\df x\df t + 2\gamma(2-\alpha)s\iint_Q\Theta \psi_\varepsilon^{1+\alpha}(D^2\psi_\varepsilon\nabla v)\cdot \nabla v\df x\df t
\end{eqnarray*}
by
\begin{equation*}
\begin{split}
&+2s\iint_Q (w_\varepsilon \nabla v\cdot\nabla w_\varepsilon)(\nabla v\cdot \nabla \xi)\df x\df t = +2\gamma\alpha(2-\alpha)s\iint_Q \Theta \psi_\varepsilon^\alpha (\nabla v\cdot\nabla \psi_\varepsilon)^2\df x\df t,
\end{split}
\end{equation*}
and by $\nabla v\cdot \nabla (\nabla v\cdot \nabla \xi)=\frac{1}{2}\nabla \xi \cdot \nabla |\nabla v|^2+(D^2\xi \nabla v)\cdot \nabla v$ we have
\begin{eqnarray*}
&&+2s\iint_Q w_\varepsilon^2\nabla v\cdot \nabla (\nabla v\cdot \nabla \xi)\df x\df t\\
&&= +s\iint_Q w_\varepsilon^2\nabla \xi \cdot \nabla |\nabla v|^2\df x\df t + 2s\iint_Q w_\varepsilon^2(D^2\xi \nabla v)\cdot \nabla v\df x\df t\\
&&= +s\iint_Q \Div(|\nabla v|^2w_\varepsilon^2\nabla \xi)\df x\df t - s\iint_Q |\nabla v|^2\Div(w_\varepsilon^2\nabla\xi)\df x\df t\\
&&\hspace{4.5mm}+2s\iint_Q w_\varepsilon^2(D^2\xi\nabla v)\cdot \nabla v\df x\df t\\
&&= +\gamma(2-\alpha)s\iint_{\partial Q}\Theta\psi_\varepsilon^\alpha \left(\frac{\partial v}{\partial \nu}\right)^2(x\cdot \nu)\df S\df t\\
&&\hspace{4.5mm}-\gamma(2-\alpha)(1+\alpha)s\iint_Q \Theta \psi_\varepsilon^\alpha |\nabla v|^2|\nabla\psi_\varepsilon|^2\df x\df t - \gamma(2-\alpha)s\iint_Q\Theta \psi_\varepsilon^\alpha|\nabla v|^2\psi_\varepsilon \Delta\psi_\varepsilon\df x\df t\\
&&\hspace{4.5mm}+2\gamma(2-\alpha)s\iint_Q\Theta \psi_\varepsilon^{1+\alpha}(D^2\psi_\varepsilon\nabla v)\cdot \nabla v\df x\df t.
\end{eqnarray*}
where $\nu$ is the outer normal derivative of $\partial\Omega$.

{\it (iii) Compute $(P_{13}v, P_{21}v)_{L^2(Q)}$.}

Indeed, by $v = 0$ on $\partial Q$, we have
\begin{eqnarray*}
&&(P_{13}v,P_{21}v)_{L^2(Q)}\\
&&= -s\iint_Q v\Div(w_\varepsilon\nabla\xi)\Div(w_\varepsilon \nabla v)\df x\df t\\
&&= +s\iint_Q w_\varepsilon \nabla v \cdot \nabla (v\Div(w_\varepsilon\nabla \xi))\df x\df t\\
&&= +s\iint_Q (w_\varepsilon \nabla v\cdot\nabla v)\Div(w_\varepsilon\nabla \xi)\df x\df t + s\iint_Q \left\{w_\varepsilon \nabla v\cdot \nabla [\Div(w_\varepsilon\nabla \xi)]\right\}v\df x\df t\\
&&= +\gamma(2-\alpha)s\iint_Q \Theta \psi_\varepsilon^\alpha |\nabla v|^2|\nabla\psi_\varepsilon|^2\df x\df t + \gamma(2-\alpha)s\iint_Q \Theta \psi_\varepsilon^\alpha |\nabla v|^2\psi_\varepsilon\Delta\psi_\varepsilon\df x\df t\\
&&\hspace{4.5mm}+s\iint_Q \left\{w_\varepsilon \nabla v\cdot \nabla [\Div(w_\varepsilon\nabla \xi)]\right\}v\df x\df t.
\end{eqnarray*}

{\it (iv): Compute $(P_{11}v, P_{22}v)_{L^2(Q)}$.}

Indeed, by i), we have
\begin{equation*}
\begin{split}
(P_{11}v, P_{22}v)_{L^2(Q)}
&= -s\iint_Q \xi_tvv_t\df x\df t = -\frac{1}{2}s\iint_{Q}\xi_t\partial_t v^2\df x\df t\\
&= +\frac{1}{2}s\iint_Q v^2\xi_{tt}\df x\df t = +\frac{1}{2}\gamma s\iint_Q\Theta'' v^2\left[-2m^{2-\alpha}+\psi_\varepsilon^{2-\alpha}\right]\df x\df t.
\end{split}
\end{equation*}

{\it (v):  Compute $(P_{12}v, P_{22}v)_{L^2(Q)}$.}

Indeed, by $v = 0$ on $\partial Q$, we have
\begin{eqnarray*}
&&(P_{12}v, P_{22}v)_{L^2(Q)}\\
&&= +2s^2\iint_Q(w_\varepsilon\nabla v\cdot \nabla \xi)(\xi_tv)\df x\df t = - s^2\iint_Q v^2\Div(\xi_tw_\varepsilon\nabla\xi)\df x\df t\\
&&= - s^2\iint_Q v^2w_\varepsilon \nabla\xi\cdot\nabla\xi_t\df x\df t - s^2\iint_Q v^2\xi_t \Div(w_\varepsilon \nabla\xi)\df x\df t\\
&&= - s^2\iint_Q v^2\xi_t \Div(w_\varepsilon \nabla\xi)\df x\df t - \gamma^2(2-\alpha)^2s^2\iint_Q \Theta'\Theta v^2\psi_\varepsilon^{2-\alpha}|\nabla\psi_\varepsilon|^2\df x\df t.
\end{eqnarray*}

{\it Compute $(P_{13}v, P_{22}v)_{L^2(Q)}$.}

Indeed, we have
\begin{equation*}
\begin{split}
&(P_{13}v, P_{22}v)_{L^2(Q)}=+s^2\iint_Q v^2\xi_t\Div(w_\varepsilon\nabla\xi)\df x\df t.
\end{split}
\end{equation*}

{\it (vi): Compute $(P_{11}v, P_{23}v)_{L^2(Q)}$.}

Indeed, we have
\begin{equation*}
\begin{split}
(P_{11}v, P_{23}v)_{L^2(Q)}
&= +s^2\iint_Q vv_tw_\varepsilon \nabla\xi\cdot\nabla\xi\df x\df t = +\frac{1}{2}s^2\iint_Q (\partial_t v^2)w_\varepsilon \nabla\xi\cdot\nabla\xi\df x\df t\\
&= - \gamma^2(2-\alpha)^2s^2\iint_Q \Theta'\Theta v^2\psi_\varepsilon^{2-\alpha}|\nabla\psi_\varepsilon|^2\df x\df t.
\end{split}
\end{equation*}

{\it (viii):  Compute $(P_{13}v, P_{23}v)_{L^2(Q)}$.}

Indeed, we have
\begin{equation*}
\begin{split}
&(P_{13}v, P_{23}v)_{L^2(Q)}=-s^3\iint_Q v^2(w_\varepsilon\nabla\xi\cdot\nabla\xi)\Div(w_\varepsilon\nabla\xi)\df x\df t.
\end{split}
\end{equation*}

{\it (ix):  Overall}, we have
\begin{equation*}
\begin{split}
H_1(Q):\hspace{-1mm}&=(P_1v,P_{21}v)_{L^2(Q)}\\
&=-\gamma(2-\alpha)s\iint_{\partial Q}\Theta \psi_\varepsilon^\alpha \left(\frac{\partial v}{\partial \nu}\right)^2(x\cdot \nu)\df S\df t-\gamma\alpha (2-\alpha) s\iint_Q\Theta \psi_\varepsilon^\alpha |\nabla v|^2|\nabla\psi_\varepsilon|^2\df x\df t\\
&\hspace{4.5mm}+2\gamma(2-\alpha)s\iint_Q\Theta \psi_\varepsilon^\alpha (\nabla v\cdot\nabla \psi_\varepsilon)^2\df x\df t+2\gamma(2-\alpha)s\iint_Q\Theta \psi_\varepsilon^{1+\alpha}(D^2\psi_\varepsilon\nabla v)\cdot \nabla v\df x\df t\\
&\hspace{4.5mm}+s\iint_Q \left\{w_\varepsilon \nabla v\cdot \nabla [\Div(w_\varepsilon\nabla \xi)]\right\}v\df x\df t,
\end{split}
\end{equation*}
and
\begin{equation*}
\begin{split}
H_2(Q):\hspace{-1mm}&=(P_1v, P_{22}v)_{L^2(Q)}+(P_{11}v,P_{23}v)_{L^2(Q)}\\
&=+\frac{1}{2}\gamma s\iint_Q\Theta'' v^2\left[-2m^{2-\alpha}+\psi_\varepsilon^{2-\alpha}\right]\df x\df t-2\gamma^2(2-\alpha)^2s^2\iint_Q \Theta'\Theta v^2 \psi_\varepsilon^{2-\alpha}|\nabla\psi_\varepsilon|^2\df x\df t,
\end{split}
\end{equation*}
and
\begin{equation*}
\begin{split}
H_3(Q):\hspace{-1mm}
&=(P_1v, P_{23}v)_{L^2(Q)}-(P_{11}v,P_{23}v)\\
&=+\gamma^3(2-\alpha)^4s^3\iint_Q \Theta^3 v^2\psi_\varepsilon^{2-\alpha}|\nabla\psi_\varepsilon|^4\df x\df t \\
&\hspace{4.5mm}+2\gamma^3(2-\alpha)^3s^3\iint_Q\Theta^3v^2\psi_\varepsilon^{3-\alpha}(D^2\psi_\varepsilon\nabla\psi_\varepsilon)\cdot\nabla\psi_\varepsilon\df x\df t.
\end{split}
\end{equation*}

\subsection{Estimations}

  For the approximate Proposition \ref{12.23.P1}, it is necessary to conduct the estimation in several distinct parts. These parts differ from those in the classical Carleman estimate due to the presence of an approximation process, which constitutes a major contribution of this paper.

This process is further divided into multiple steps.

{\it Step 1: Analysis of the part on $Q_\varepsilon^c := (\Omega \setminus B_\varepsilon) \times (0,T)$ with $\psi_\varepsilon = |x|$.}

(i) From equation \eqref{12.18.2} and $\Div(w_\varepsilon \nabla \xi) = \gamma N(2 - \alpha) \Theta$ on $Q_\varepsilon^c$, we deduce $\nabla[\Div(w_\varepsilon \nabla \xi)] = 0$ on $Q_\varepsilon^c$. Consequently,
$$
\begin{aligned}
H_1(Q_\varepsilon^c)
&= -\gamma(2 - \alpha)s \iint_{\partial Q} \Theta |x|^\alpha \left(\frac{\partial v}{\partial \nu}\right)^2 (x \cdot \nu) \, dS \, dt - \gamma \alpha(2 - \alpha)s \iint_{Q_\varepsilon^c} \Theta |x|^\alpha |\nabla v|^2 \, dx \, dt \\
&\quad + 2\gamma (2 - \alpha)s \iint_{Q_\varepsilon^c} \Theta |x|^{\alpha - 2} (\nabla v \cdot x)^2 \, dx \, dt \\
&\quad + 2\gamma (2 - \alpha)s \iint_{Q_\varepsilon^c} \Theta |x|^\alpha |\nabla v|^2 \, dx \, dt - 2\gamma (2 - \alpha)s \iint_{Q_\varepsilon^c} \Theta |x|^\alpha (\nabla v \cdot x)^2 \, dx \, dt \\
&= -\gamma(2 - \alpha)s \iint_{\partial Q} \Theta |x|^\alpha \left(\frac{\partial v}{\partial \nu}\right)^2 (x \cdot \nu) \, dS \, dt + \gamma(2 - \alpha)^2s \iint_{Q_\varepsilon^c} \Theta |x|^\alpha |\nabla v|^2 \, dx \, dt.
\end{aligned}
$$
(ii) Given that
\begin{equation}\label{12.21.1}
|\Theta' \Theta| \leq 12T \Theta^{\frac{9}{4}} \leq \frac{C}{T^5} \Theta^3,  \;
|\Theta''| \leq (60T + 8T^2) \Theta^{\frac{3}{2}} \leq \frac{C(1 + T)}{T^{11}},
\end{equation}
where $C > 0$ is an absolute constant, we obtain
$$
\left| -2\gamma^2(2 - \alpha)^2s^2 \iint_{Q_\varepsilon^c} \Theta' \Theta v^2 |x|^{2 - \alpha} \, dx \, dt \right| \leq C\gamma^2s^2 \iint_{Q_\varepsilon^c} \Theta^3 v^2 |x|^{2 - \alpha} \, dx \, dt,
$$
and
$$
\begin{aligned}
\left| \frac{1}{2}\gamma s \iint_{Q_\varepsilon^c} \Theta'' v^2 |x|^{2 - \alpha} \, dx \, dt \right| &\leq C\gamma s \iint_{Q_\varepsilon^c} \Theta^3 v^2 |x|^{2 - \alpha} \, dx \, dt,
\end{aligned}
$$
and
\begin{eqnarray*}
&&\left| \frac{1}{2}\gamma s \iint_{Q_\varepsilon^c} \Theta'' v^2 (-2m^{2 - \alpha}) \, dx \, dt \right| \\
&&\leq C\gamma s \iint_{Q_\varepsilon^c} \Theta^{\frac{3}{2}} v^2 \, dx \, dt = Cs \iint_{Q_\varepsilon^c} \left(\gamma \Theta |x|^{1 - \frac{\alpha}{2}} v\right) \left(\Theta^{\frac{1}{2}} |x|^{\frac{\alpha}{2} - 1} v\right) \, dx \, dt \\
&&\leq C\gamma^2 s \iint_{Q_\varepsilon^c} \Theta^2 v^2 |x|^{2 - \alpha} \, dx \, dt + C s \iint_{Q_\varepsilon^c} \Theta |x|^{\alpha - 2} v \, dx \, dt \\
&&\leq C\gamma^2s \iint_{Q_\varepsilon^c} \Theta^3 v^2 |x|^{2 - \alpha} \, dx \, dt + Cs \iint_{Q} \Theta \psi_\varepsilon^\alpha |\nabla v|^2 |\nabla \psi_\varepsilon|^2 \, dx \, dt
\end{eqnarray*}
by Lemma \ref{08.16.L2} and $|\nabla \psi_\varepsilon(x)| = 1$ on $\Omega \setminus B_R$.
Thus,
\begin{eqnarray*}
H_2(Q_\varepsilon^c) &&= \frac{1}{2}\gamma s \iint_{Q_\varepsilon^c} \Theta'' v^2 [-2m^{2 - \alpha} + |x|^{2 - \alpha}] \, dx \, dt - 2\gamma^2(2 - \alpha)^2s^2 \iint_{Q_\varepsilon^c} \Theta' \Theta v^2 |x|^{2 - \alpha} \, dx \, dt \\
&&\geq -C\gamma^2 s \iint_{Q_\varepsilon^c} \Theta^3 v^2 |x|^{2 - \alpha} \, dx \, dt - Cs \iint_{Q} \Theta \psi_\varepsilon^{\alpha} |\nabla v|^2 |\nabla \psi_\varepsilon|^2 \, dx \, dt \\
&&\quad -C\gamma s \iint_{Q_\varepsilon^c} \Theta^3 v^2 |x|^{2 - \alpha} \, dx \, dt - C\gamma^2s^2 \iint_{Q_\varepsilon^c} \Theta^3 v^2 |x|^{2 - \alpha} \, dx \, dt \\
&&\geq -C\gamma^2s^2 \iint_{Q_\varepsilon^c} \Theta^3 v^2 |x|^{2 - \alpha} |\nabla \psi_\varepsilon|^4 \, dx \, dt - Cs \iint_{Q} \Theta \psi_\varepsilon^{\alpha} |\nabla v|^2 |\nabla \psi_\varepsilon|^2 \, dx \, dt.
\end{eqnarray*}
(iii) Direct computation yields
\begin{eqnarray*}
H_3(Q_\varepsilon^c)
&&= \gamma^3(2 - \alpha)^4s^3 \iint_{Q_\varepsilon^c} \Theta^3 v^2 |x|^{2 - \alpha} \, dx \, dt \\
&&\quad + 2\gamma^3(2 - \alpha)^3s^3 \iint_{Q_\varepsilon^c} \Theta^3 v^2 |x|^{2 - \alpha} \, dx \, dt - 2\gamma^3(2 - \alpha)^3s^3 \iint_{Q_\varepsilon^c} \Theta^3 v^2 |x|^{2 - \alpha} \, dx \, dt \\
&&= \gamma^3 (2 - \alpha)^4s^3 \iint_{Q_\varepsilon^c} \Theta^3 v^2 |x|^{2 - \alpha} \, dx \, dt.
\end{eqnarray*}

{\it Step 2: Analysis of  the part on $Q_\e:=B_\e\times(0,T)$ with  $\psi_\e$ is in the form \eqref{12.20.1}.}

 (i) Observe that
\begin{equation}\label{12.20.2}
\begin{split}
\psi_\varepsilon\left(\frac{3}{2\varepsilon}-\frac{1}{2\varepsilon^3}|x|^2\right)-|\nabla \psi_\varepsilon|^2=\frac{3}{16\varepsilon^6}(\varepsilon^2-|x|^2)^2(3\varepsilon^2-|x|^2)\geq 0 \quad \text{for }|x|\leq \varepsilon,
\end{split}
\end{equation}
and $\nabla\psi_\varepsilon=\left(\frac{3}{2\varepsilon}-\frac{1}{2\varepsilon^3}|x|^2\right)x$, and
\begin{equation}\label{12.20.3}
\begin{split}
\left(\frac{3}{2\varepsilon}-\frac{1}{2\varepsilon^3}|x|^2\right)^2-\frac{1}{\varepsilon^3}\psi_\varepsilon
&=\frac{3}{8\varepsilon^6}(\varepsilon^2-|x|^2)(5\varepsilon^2-|x|^2)\geq 0 \quad \text{for }|x|\leq \varepsilon.
\end{split}
\end{equation}
From
\begin{equation*}
\begin{split}
&2(D^2\psi_\varepsilon\nabla \psi_\varepsilon)+\Delta\psi_\varepsilon \nabla \psi_\varepsilon+\psi_\varepsilon\nabla (\Delta \psi_\varepsilon)\\
&=\frac{3}{8\varepsilon^6}x\left(10\varepsilon^4+5N\varepsilon^4-24\varepsilon^2|x|^2-6N\varepsilon^2|x|^2+6|x|^4+N|x|^4\right) \\
&=\nabla\psi_\varepsilon \frac{3}{4\varepsilon^3(3\varepsilon^2-|x|^2)}\left(10\varepsilon^4+5N\varepsilon^4-24\varepsilon^2|x|^2-6N\varepsilon^2|x|^2+6|x|^4+N|x|^4\right)
\end{split}
\end{equation*}
and
\begin{equation*}
-\frac{3}{\varepsilon}\leq \Xi:= \frac{3}{4\varepsilon^3(3\varepsilon^2-|x|^2)}\left(10\varepsilon^4+5N\varepsilon^4-24\varepsilon^2|x|^2-6N\varepsilon^2|x|^2+6|x|^4+N|x|^4\right)\leq \frac{5(N+2)}{4\varepsilon},
\end{equation*}
and \eqref{12.21.2}, we obtain
\begin{eqnarray*}
&&\left|s\iint_{Q_\varepsilon} \left\{w_\varepsilon \nabla v\cdot \nabla [\Div(w_\varepsilon\nabla \xi)]\right\}v\,dx\,dt\right|\\
&&=\gamma(2-\alpha)s\iint_{Q_\varepsilon}\Theta \left[\psi_\varepsilon^\alpha\nabla v\cdot \nabla \left(|\nabla \psi_\varepsilon|^2+\psi_\varepsilon\Delta\psi_\varepsilon\right)\right]v\,dx\,dt\\
&&=\gamma(2-\alpha)s\iint_{Q_\varepsilon}\Theta \psi_\varepsilon^\alpha\nabla v\cdot \left[2(D^2\psi_\varepsilon\nabla \psi_\varepsilon)+\Delta\psi_\varepsilon \nabla \psi_\varepsilon+\psi_\varepsilon\nabla (\Delta \psi_\varepsilon)\right]v\,dx\,dt\\
&&=\gamma(2-\alpha)s\iint_{Q_\varepsilon}\Theta (\psi_\varepsilon^\alpha \nabla v\cdot \nabla\psi_\varepsilon) \Xi v\,dx\,dt\\
&&\leq \frac{1}{2}\gamma (2-\alpha)^2s\iint_{Q_\varepsilon} \Theta \psi_\varepsilon^\alpha |\nabla v|^2|\nabla\psi_\varepsilon|^2\,dx\,dt+C\gamma \varepsilon^{\alpha-2}s\iint_{Q_\varepsilon} \Theta  v^2\,dx\,dt,
\end{eqnarray*}
where $C>0$ is an absolute constant.
Hence,
\begin{eqnarray*}
H_1(Q_\varepsilon)
&=&-\gamma\alpha(2-\alpha)s\iint_{Q_\varepsilon}\Theta \psi_\varepsilon^\alpha |\nabla v|^2|\nabla \psi_\varepsilon|^2\,dx\,dt+2\gamma (2-\alpha)s\iint_{Q_\varepsilon} \Theta \psi_\varepsilon^\alpha (\nabla v\cdot \nabla \psi_\varepsilon)^2\,dx\,dt\\
&&\hspace{4.5mm}+2\gamma (2-\alpha)s\iint_{Q_\varepsilon}\Theta \psi_\varepsilon^{1+\alpha}\left(\frac{3}{2\varepsilon}-\frac{1}{2\varepsilon^3}|x|^2\right)|\nabla v|^2\,dx\,dt\\
&&\hspace{4.5mm}-2\gamma(2-\alpha)s\iint_{Q_\varepsilon}\Theta \psi_\varepsilon^{1+\alpha}\frac{1}{\varepsilon^3}(\nabla v\cdot x)^2\,dx\,dt\\
&&\hspace{4.5mm}+s\iint_{Q_\varepsilon} \left\{w_\varepsilon \nabla v\cdot \nabla [\Div(w_\varepsilon\nabla \xi)]\right\}v\,dx\,dt\\
&=&\gamma(2-\alpha)^2s\iint_{Q_\varepsilon}\Theta\psi_\varepsilon^\alpha|\nabla v|^2|\nabla\psi_\varepsilon|^2\,dx\,dt\\
&&\hspace{4.5mm}+2\gamma (2-\alpha)s\iint_{Q_\varepsilon} \Theta \psi_\varepsilon^{\alpha}|\nabla v|^2\left[\psi_\varepsilon\left(\frac{3}{2\varepsilon}-\frac{1}{2\varepsilon^3}|x|^2\right)-|\nabla\psi_\varepsilon|^2\right]\,dx\,dt\\
&&\hspace{4.5mm}+2\gamma(2-\alpha)s\iint_{Q_\varepsilon}\Theta \psi_\varepsilon^\alpha (\nabla v\cdot x)^2 \left[\left(\frac{3}{2\varepsilon}-\frac{1}{2\varepsilon^3}|x|^2\right)^2-\frac{1}{\varepsilon^3}\psi_\varepsilon\right]\,dx\,dt\\
&&\hspace{4.5mm}+s\iint_{Q_\varepsilon} \left\{w_\varepsilon \nabla v\cdot \nabla [\Div(w_\varepsilon\nabla \xi)]\right\}v\,dx\,dt\\
&\geq& \frac{1}{2}\gamma (2-\alpha)^2s\iint_{Q_\varepsilon}\Theta\psi_\varepsilon^\alpha|\nabla v|^2|\nabla\psi_\varepsilon|^2\,dx\,dt-C\gamma \varepsilon^{\alpha-2}s\iint_{Q_\varepsilon}\Theta v^2\,dx\,dt.
\end{eqnarray*}

(ii) From \eqref{12.21.1}, we get
\begin{equation*}
\begin{split}
H_2(Q_\e)\geq -C\gamma s\iint_{Q_\e} \Theta ^3 v^2dxdt-C\gamma^2s^2\iint_{Q_\e}\Theta^3 \psi_\e^{2-\alpha}v^2dxdt.
\end{split}
\end{equation*}
(iii) Since
\begin{equation*}
\begin{split}
\psi_\e (D^2\psi_\e\nabla \psi_\e)\cdot\nabla \psi_\e
&=+\psi_\e\left[\left(\frac{1}{2\varepsilon}-\frac{1}{2\varepsilon^3}|x|^2\right)|\nabla\psi_\e|^2-\frac{1}{\varepsilon^3}(\nabla\psi_\e\cdot x)^2\right]\\
&=\frac{3}{64\varepsilon^{12}}|x|^2(\varepsilon^2-|x|^2)(3\varepsilon^2-|x|^2)^2(3\varepsilon^4+6\varepsilon^2|x|^2-|x|^4)\geq 0 \text{  for } |x|\leq \varepsilon,
\end{split}
\end{equation*}
we have
\begin{equation*}
\begin{split}
H_3(Q_\e) \geq + \gamma^3(2-\alpha)^4s^3\iint_{Q_\e}\Theta^3v^2\psi_\e^{2-\alpha}|\nabla\psi_\e|^4dxdt.
\end{split}
\end{equation*}

 {\it Step 3.} Given $|\nabla\psi_\varepsilon(x)|=1$ on $\Omega\setminus B_R$, $(P_1v,P_2v)_{L^2(Q)}\leq \frac{1}{2}\|e^{s\xi}g\|_{L^2(Q)}$, along with the results from Step 1 and Step 2, we derive
\begin{eqnarray*}
&&+\frac{1}{2}\gamma(2-\alpha)^2s\iint_Q \Theta \psi_\varepsilon^\alpha |\nabla v|^2|\nabla \psi_\varepsilon|^2dxdt\\
&&+\gamma^3(2-\alpha)^4s^3\iint_{Q}\Theta^3v^2\psi_\varepsilon^{2-\alpha}|\nabla\psi_\varepsilon|^4dxdt\\
&&-C\gamma^2s^2\iint_{Q_\varepsilon^c}\Theta ^3v^2\psi_\varepsilon^{2-\alpha}|\nabla\psi_\varepsilon|^4dxdt-Cs\iint_Q\Theta \psi_\varepsilon^\alpha |\nabla v|^2|\nabla \psi_\varepsilon|^2dxdt\\
&&\leq +\frac{1}{2}\iint_Q e^{2s\xi}g^2dxdt+\gamma(2-\alpha)s\iint_{\partial Q}\Theta \psi_\varepsilon^\alpha \left(\frac{\partial v}{\partial \nu}\right)^2(x\cdot \nu)dSdt\\
&&\hspace{4.5mm}+C\gamma \varepsilon^{\alpha-2}s\iint_{Q_\varepsilon}\Theta v^2dxdt+C\gamma s\iint_{Q_\varepsilon} \Theta^3v^2dxdt+C\gamma^2s^2\iint_{Q_\varepsilon}\Theta^3 \psi_\varepsilon^{2-\alpha}v^2dxdt,
\end{eqnarray*}
By choosing $\gamma\geq 1$ sufficiently large and $s\geq 1$ such that
\begin{equation*}
\frac{1}{2}\gamma (2-\alpha)^2>C+1,   \gamma (2-\alpha)^4>C+1,
\end{equation*}
we obtain
\begin{equation*}
\begin{split}
&+s\iint_Q \Theta \psi_\varepsilon^\alpha |\nabla v|^2|\nabla \psi_\varepsilon|^2dxdt+s^3\iint_{Q}\Theta^3v^2\psi_\varepsilon^{2-\alpha}|\nabla\psi_\varepsilon|^4dxdt\\
&\leq +C\iint_Q e^{2s\xi}g^2dxdt+C s\iint_{\partial Q}\Theta \psi_\varepsilon^\alpha \left(\frac{\partial v}{\partial \nu}\right)^2(x\cdot \nu)dSdt \\
&\hspace{4.5mm}+Cs^2\iint_{Q_\varepsilon}\Theta^3v^2dxdt+C \varepsilon^{\alpha-2}s\iint_{Q_\varepsilon}\Theta v^2dxdt
\end{split}
\end{equation*}
in light of $|\nabla\psi_\varepsilon(x)|=1$ on $\Omega\setminus B_R$, where the constants $C>0$ are absolute.
Now, consider the transformation $u = e^{-s\xi}v$, which leads to
\begin{equation*}
\begin{split}
\nabla u &= e^{-s\xi}\nabla v + (-s\nabla \xi)ve^{-s\xi} \\
&= e^{-s\xi}\left[\nabla v - \gamma(2-\alpha)s\Theta \psi_\varepsilon^{1-\alpha} v\nabla \psi_\varepsilon\right],
\end{split}
\end{equation*}
and thus
\begin{eqnarray*}
&&+s\iint_Q\Theta \psi_\varepsilon^\alpha |\nabla u|^2|\nabla \psi_\varepsilon|^2 e^{2s\xi}\,dx\,dt + s^3\iint_Q \Theta^3 \psi_\varepsilon^{2-\alpha}u^2|\nabla\psi_\varepsilon|^4e^{2s\xi}\,dx\,dt \\
&&= +s\iint_Q\Theta\psi_\varepsilon^\alpha \left|\nabla v - \gamma(2-\alpha)s\Theta \psi_\varepsilon^{1-\alpha}v\nabla\psi_\varepsilon\right|^2|\nabla\psi_\varepsilon|^2\,dx\,dt + s^3\iint_Q\Theta^3\psi_\varepsilon^{2-\alpha}v^2|\nabla\psi_\varepsilon|^4\,dx\,dt \\
&&\leq +s\iint_Q\Theta \psi_\varepsilon^\alpha |\nabla v|^2|\nabla \psi_\varepsilon|^2\,dx\,dt + Cs^3\iint_Q \Theta^3 \psi_\varepsilon^{2-\alpha}v^2|\nabla \psi_\varepsilon|^4\,dx\,dt \\
&&\leq +C\iint_Q e^{2s\xi}g^2\,dx\,dt + Cs\iint_{\partial Q}\Theta \psi_\varepsilon^\alpha \left(\frac{\partial v}{\partial \nu}\right)^2(x\cdot\nu)\,dS\,dt \\
&&\hspace{4.5mm} + Cs^2\iint_{Q_\varepsilon}\Theta^3 v^2\,dx\,dt + C\varepsilon^{\alpha-2}s^2\iint_{Q_\varepsilon}\Theta v^2\,dx\,dt \\
&&= +C\iint_Q e^{2s\xi}g^2\,dx\,dt + Cs\iint_{\partial Q}\Theta \psi_\varepsilon^\alpha \left(\frac{\partial u}{\partial \nu}\right)^2(x\cdot\nu)e^{2s\xi}\,dS\,dt \\
&&\hspace{4.5mm} + Cs^2\iint_{Q_\varepsilon}\Theta^3 u^2e^{2s\xi}\,dx\,dt + C\varepsilon^{\alpha-2}s^2\iint_{Q_\varepsilon}\Theta u^2e^{2s\xi}\,dx\,dt,
\end{eqnarray*}
owing to the fact that
\begin{equation*}
\nabla v = \nabla (e^{s\xi}u) = e^{s\xi}\left[\nabla u + su\nabla \xi\right] = e^{s\xi}\nabla u \text{  on } \partial Q.
\end{equation*}
\begin{theorem}\label{12.21.T1}
Let $T > 0$ and $\varepsilon \in (0, \frac{R}{2})$. Then, there exist constants $C > 0$, independent of $\varepsilon > 0$, such that for every solution $u$ of \eqref{12.15.1},  and all  $s \geq 1$, there holds
\begin{equation*}
\begin{split}
&+s\iint_Q\Theta \psi_\varepsilon^\alpha |\nabla u|^2|\nabla \psi_\varepsilon|^2 e^{2s\xi}\,dx\,dt + s^3\iint_Q \Theta^3 \psi_\varepsilon^{2-\alpha}u^2|\nabla\psi_\varepsilon|^4e^{2s\xi}\,dx\,dt \\
&\leq +C\iint_Q e^{2s\xi}g^2\,dx\,dt + Cs\iint_{\partial Q}\Theta \psi_\varepsilon^\alpha \left(\frac{\partial u}{\partial \nu}\right)^2(x\cdot\nu)e^{2s\xi}\,dS\,dt \\
&\hspace{4.5mm} + Cs^2\iint_{Q_\varepsilon}\Theta^3 u^2 e^{2s\xi}\,dx\,dt + C\varepsilon^{\alpha-2}s^2\iint_{Q_\varepsilon}\Theta u^2e^{2s\xi}\,dx\,dt.
\end{split}
\end{equation*}
\end{theorem}
\begin{remark}
We note that $(x\cdot \nu)$ may assume negative values on $\partial Q$.
\end{remark}

 \subsection{Approximation}

 \begin{lemma}\label{01.18.L1}
Let $0 < \varepsilon \ll 1$. Then the function $w_\varepsilon$ is an $A_{1+\frac{2}{N}}$ weight. Moreover, there exists a constant $C > 0$, independent of $\varepsilon > 0$, such that
\begin{equation*}
c\left(w_\varepsilon, 1+\frac{1}{N}\right) \leq C,
\end{equation*}
where $c(w_\varepsilon, 1+\frac{1}{N})$ denotes the $A_{1+\frac{1}{N}}$ constant of $w_\varepsilon$.
\end{lemma}
\begin{proof}
It is clear that $w_\varepsilon$ is an $A_1$ weight since
\begin{equation*}
\frac{1}{|K|}\int_K w_\varepsilon(x) \, dx \leq \frac{4m^\alpha}{\varepsilon} \essinf_K w,
\end{equation*}
and $\essinf_K w_\varepsilon \geq \frac{\varepsilon}{4}$. By the property $A_{p_1} \subseteq A_{p_2}$ for $1 \leq p_1 < p_2$, it follows that $w_\varepsilon$ is an $A_{1+\frac{2}{N}}$ weight.
Next, we show that $\{c(w_\varepsilon, 1+\frac{1}{N}) \colon \varepsilon \in (0,1)\}$ is uniformly bounded and depends only on $c(w,1+\frac{1}{N})$.

There are three  cases.

\textbf{Case 1: $0 < a \ll \varepsilon$.}
Taking $y=(y_1,\cdots,y_N)\in \Omega$, consider $K=\prod_{i=1}^N(y_i-a,y_i+a)$. For sufficiently small $a$, since $w_\varepsilon \geq \frac{\varepsilon}{4}$, we have $\frac{1}{2}w_\varepsilon(y) \leq w_\varepsilon(x) \leq 2w_\varepsilon(y)$ for all $x \in K$. Thus,
\begin{equation*}
\begin{split}
\frac{1}{|K|}\int_K w_\varepsilon(x) \, dx
&\leq 2 (2a)^{-N}\int_{B_{\sqrt{N}a}} w_\varepsilon(y) \, dx = 2(2a)^{-N} w_\varepsilon(y)|B_{\sqrt{N}a}| = \frac{|B_1|N^{\frac{N}{2}}}{2^{N-1}} w_\varepsilon(y),
\end{split}
\end{equation*}
and
\begin{equation*}
\begin{split}
\left(\frac{1}{|K|}\int_K w_\varepsilon(x)^{-\frac{N}{2}} \, dx\right)^{\frac{2}{N}}
&\leq 2w_\varepsilon(y)^{-1}\left(\frac{|B_1|N^{\frac{N}{2}}}{2^N}\right)^{\frac{2}{N}} = w_\varepsilon(y)^{-1}\frac{|B_1|^{\frac{2}{N}}N}{2}.
\end{split}
\end{equation*}
These imply that
\begin{equation*}
\begin{split}
\left(\frac{1}{|K|}\int_K w_\varepsilon(x) \, dx\right)\left(\frac{1}{|K|}\int_K w_\varepsilon(x)^{-\frac{N}{2}} \, dx\right)^{\frac{2}{N}} \leq \frac{1}{2^N}|B_1|^{1+\frac{2}{N}}N^{1+\frac{N}{2}}.
\end{split}
\end{equation*}
\textbf{Case 2: $a \approx \varepsilon$.}
Taking $y=(y_1,\cdots,y_N)\in \Omega$, consider $K_y=\prod_{i=1}^N(y_i-a,y_i+a)$ for $a \approx \varepsilon$. It is evident that
\begin{eqnarray*}
&&\left(\frac{1}{|K_y|}\int_{K_y} w_\varepsilon(x) \, dx\right)\left(\frac{1}{|K_y|}\int_{K_y} w_\varepsilon(x)^{-\frac{N}{2}} \, dx\right)^{\frac{2}{N}} \\
&&\leq (3\varepsilon)^\alpha \frac{1}{4\varepsilon^2}\left(\int_{B_{\sqrt{N}\varepsilon}}|x|^{-\frac{N\alpha}{2}} \, dx\right)^{\frac{2}{N}} = (3\varepsilon)^\alpha \frac{1}{4\varepsilon^2} |\partial B_1|^{\frac{2}{N}}\left(\int_0^{\sqrt{N}\varepsilon} r^{N-1-\frac{N\alpha}{2}} \, dr\right)^{\frac{2}{N}} \\
&&\leq \frac{3^\alpha}{2^{2-\frac{2}{N}}}|\partial B_1|^{\frac{2}{N}} (2-\alpha)^{-\frac{2}{N}}N^{1-\frac{\alpha}{2}-\frac{2}{N}}
\end{eqnarray*}
for $|y| \leq 2\varepsilon$. The other cases are similar to the case of $|x|^\alpha$.

\textbf{Case 3: $\varepsilon \ll a$.}
Taking $y=(y_1,\cdots,y_N)\in \Omega$, consider $K_y=\prod_{i=1}^N(y_i-a,y_i+a)$ for $\varepsilon \ll a$. Then,
\begin{equation*}
\begin{split}
&\left(\frac{1}{|K_y|}\int_{K_y} w_\varepsilon(x) \, dx\right)\left(\frac{1}{|K_y|}\int_{K_y} w_\varepsilon(x)^{-\frac{N}{2}} \, dx\right)^{\frac{2}{N}} \\
&\approx \left(\frac{1}{|K_y|}\int_{K_y} w(x) \, dx\right)\left(\frac{1}{|K_y|}\int_{K_y} w(x)^{-\frac{N}{2}} \, dx\right)^{\frac{2}{N}} \leq c\left(w,1+\frac{2}{N}\right).
\end{split}
\end{equation*}
Overall, we conclude that $c(w_\varepsilon, 1+\frac{2}{N}) \leq C$ for all $\varepsilon > 0$, where the constant $C > 0$ depends only on $w$ (i.e., $\alpha, N$ and $\Omega$). Moreover, the constant $C > 0$ is independent of $\varepsilon > 0$.
\end{proof}

  Denote
\begin{equation*}
\psi_k = \psi_\varepsilon, \quad \xi_k = \Theta \eta_\varepsilon, \quad \widehat{\varphi}_k = \widehat{\varphi}_\varepsilon, \quad Q_k = Q_\varepsilon
\end{equation*}
where $\varepsilon = \frac{1}{k}$ and $k \in \mathbb{N}$. Additionally, define $\xi_0 = \Theta \gamma(-2m^{2-\alpha} + \psi^{2-\alpha})$ with $\psi = |x|$.
By Theorem \ref{12.21.T1}, for $g \in L^2(Q;w^{-1})$, $k \in \mathbb{N}$, and $\widehat{\varphi}_k(T) = \varphi_T \in L^2(\Omega)$, we have
\begin{equation}\label{12.22.1}
\begin{split}
&s\iint_Q \Theta \psi_k^\alpha |\nabla \widehat{\varphi}_k|^2|\nabla \psi_k|^2 e^{2s\xi_k} \, dx \, dt + s^3\iint_Q \Theta^3\psi_k^{2-\alpha} \widehat{\varphi}_k^2|\nabla\psi_k|^4e^{2s\xi_k} \, dx \, dt\\
&\leq C\iint_Q e^{2s\xi_k}g^2 \, dx \, dt + Cs\iint_{\partial Q} \Theta \psi_k^\alpha \left(\frac{\partial \widehat{\varphi}_k}{\partial \nu}\right)^2(x \cdot \nu)e^{2s\xi_k} \, dS \, dt\\
&\quad +Cs^2\iint_{Q_k}\Theta^3\widehat{\varphi}_k^2 e^{2s\xi_k} \, dx \, dt + Ck^{2-\alpha}s^2\iint_{Q_k}\Theta \widehat{\varphi}_k^2e^{2s\xi_k} \, dx \, dt,
\end{split}
\end{equation}
where the constant $C > 0$ is independent of $k \in \mathbb{N}$.

  To ensure thorough understanding, we have divided our analysis into distinct stages.

{\it Step 1}. Assume that $\widehat{\varphi}_k\ (k \in \mathbb{N})$ is the solution of \eqref{12.15.1} on $Q = (-\beta,T+\beta)$ with $\widehat{\varphi}_k(T) = u_T$. Since $B_R \times (0,T)$ is a compact subset of $\Omega \times (-\beta,T+\beta)$, by Theorem 3.11 or Theorem 3.14 in \cite{FF} and Lemma \ref{01.18.L1}, there exists a subsequence of $\{\widehat{\varphi}_k\}_{k \in \mathbb{N}}$, still denoted by itself, such that
\begin{equation*}
\widehat{\varphi}_k \to \widehat{\varphi}_0 \quad \text{uniformly on } B_R \times (0,T).
\end{equation*}
Hence, for any $h > 0$, there exists $k_0 > 0$ such that for any $k \geq k_0$,
\begin{equation*}
|\widehat{\varphi}_k(x,t) - \widehat{\varphi}_0(x,t)| < h \quad \text{on } B_R \times (0,T),
\end{equation*}
and thus, as $k \to +\infty$, it has
\begin{equation*}
\begin{split}
k^{2-\alpha}\iint_{Q_k}\Theta  \widehat{\varphi}_k^2e^{2s\xi_k} \, dx \, dt
&\leq 2k^{2-\alpha}\iint_{Q_k}\Theta  (\widehat{\varphi}_k - \widehat{\varphi}_0)^2e^{2s\xi_k} \, dx \, dt + 2k^{2-\alpha}\iint_{Q_\varepsilon}\Theta \widehat{\varphi}_0^2e^{2s\xi_k} \, dx \, dt\\
&\leq Ck^{2-\alpha}|B_k|h^2\int_0^T\Theta e^{-C\Theta} \, dt + Ck^{2-\alpha}|B_k|\int_0^T\Theta e^{-C\Theta}\left(\frac{1}{|B_k|}\int_{B_k} \widehat{\varphi}_0^2 \, dx\right) \, dt \to 0,
\end{split}
\end{equation*}
by $N \geq 2$, $-2m^{2-\alpha} \leq \eta_k \leq -m^{2-\alpha}$, $\Theta e^{-C\Theta}$ is bounded on $Q$, and the Lebesgue differentiation theorem
\begin{equation*}
\lim_{k \to 0}\frac{1}{|B_k|}\int_{B_k}\widehat{\varphi}_0^2 \, dx \, dt = \widehat{\varphi}_0(0,t) \quad \text{for all } t \in (0,T).
\end{equation*}
Moreover, since $\Theta^3e^{-C\Theta}$ is bounded on $Q$, we have
\begin{equation*}
k^{2-\alpha}\iint_{Q_k}\Theta^3\widehat{\varphi}_k^2 e^{2s\xi_k} \, dx \, dt \to 0 \quad \text{as } k \to +\infty.
\end{equation*}
{\it Step 2}. We have
\begin{equation*}
\begin{split}
\iint_{Q_k} \Theta^3\psi_k^{2-\alpha}\widehat{\varphi}_k^2|\nabla \psi_k|^4e^{2s\xi_k} \, dx \, dt
&\leq Ck^{\alpha-2}\iint_{Q_k} \Theta^3 \widehat{\varphi}_k^2e^{2s\xi_k} \, dx \, dt\\
&\leq Ck^{\alpha-2}\iint_{Q_k}  \widehat{\varphi}_k^2 \, dx \, dt \to 0 \quad \text{as } k \to \infty
\end{split}
\end{equation*}
by Corollary \ref{08.23.C1} and Step 1. We also have
\begin{equation*}
\iint_{Q_k}\Theta^3\psi^{2-\alpha}\widehat{\varphi}_0^2e^{2s\xi_0} \, dx \, dt \leq k^{\alpha-2}\iint_{Q}\Theta^3\widehat{\varphi}_0^2e^{2s\xi_0} \, dx \, dt \to 0 \quad \text{as } k \to \infty.
\end{equation*}

{\it Step 3}. By the same argument as Step 1, $\xi_k \to \xi_0$ on $Q$ everywhere, and Corollary \ref{08.23.C1}, for each $\varphi_T \in H_0^1(\Omega) \cap H^3(\Omega)$ and $\Div(w_k \nabla \varphi_T) \in H_0^1(\Omega)$, we have
\begin{equation*}
\iint_{\partial Q}\Theta \psi_k^\alpha\left(\frac{\partial \widehat{\varphi}_k}{\partial\nu}\right)^2 (x \cdot \nu)e^{2s\xi_k} \, dS \, dt \to \iint_{\partial Q}\Theta \psi^\alpha \left(\frac{\partial\widehat{\varphi}_0}{\partial\nu}\right)^2(x \cdot \nu)e^{2s\xi_0} \, dS \, dt \mbox{ as } k \to \infty.
\end{equation*}
From $0 < \frac{1}{k} \ll R$, \eqref{12.22.1} and Corollary \ref{08.23.C1}, we have
\begin{equation*}
\begin{split}
\iint_{(\Omega\setminus B_R) \times (0,T)} \Theta \psi^\alpha |\nabla \widehat{\varphi}_0|^2e^{2s\xi_0} \, dx \, dt
&\leq \liminf_{k \to \infty}\iint_Q \Theta \psi^\alpha |\nabla \widehat{\varphi}_k|^2|\nabla \psi_k|^2e^{2s\xi_k}\chi_{\Omega\setminus B_k} \, dx \, dt\\
&\leq \liminf_{k \to \infty}\iint_Q \Theta \psi_k^\alpha |\nabla \widehat{\varphi}_k|^2|\nabla\psi_k|^2e^{2s\xi_k} \, dx \, dt.
\end{split}
\end{equation*}

   By combining \eqref{12.22.1} with Steps 1–3 and utilizing Corollary \ref{08.23.C1}, for every $\vp_T\in H_0^1(\Om)\cap H^3(\Om)$ and $\Div(w_k \nabla \vp_T)\in H_0^1(\Om)$, we derive
\begin{equation}\label{gbz5}
\begin{split}
&s\iint_{(\Om\setminus B_R)\times (0,T)} \Theta \psi^\al |\nabla \wh\vp_0|^2e^{2s\xi_0}\df x\df t+s^3\iint_Q \Theta^3\psi^{2-\al}\wh\vp_0^2e^{2s\xi_0}\df x\df t\\
&\leq C\iint_Q e^{2s\xi_0}g^2\df x\df t+Cs\iint_{\pt Q}\Theta \psi^{\al}\left(\f{\pt\wh\vp_0}{\pt\nu}\right)^2(x\cdot\nu)e^{2s\xi_0}\df S\df t,
\end{split}
\end{equation}
where the positive constant $C$ depends exclusively on $\al, N, T, R$ and $\Om$. We summarize \dref{gbz5}
into Theorem \ref{12.23.T1} following.

\begin{theorem}\label{12.23.T1}
Let $T>0$. Then, there exists a positive constant $C$ such that for every solution $\wh\vp_0$ of \eqref{12.15.1} satisfying, for all $s\geq 1$, we have
\begin{equation*}
\begin{split}
&s\iint_{(\Om\setminus B_R)\times (0,T)}\Theta \psi^\al |\nabla \wh\vp_0|^2 e^{2s\xi_0}\df x\df t+s^3\iint_Q \Theta^3 \psi^{2-\al}\wh\vp_0^2 e^{2s\xi_0}\df x\df t\\
&\leq C\iint_Q e^{2s\xi_0}g^2\df x\df t+Cs\iint_{\pt Q}\Theta \psi^\al \left(\f{\pt \wh\vp_0}{\pt \nu}\right)^2(x\cdot\nu)e^{2s\xi_0}\df S\df t.
\end{split}
\end{equation*}
\end{theorem}
{\it Step 5}. Choose $\zeta\in C^\iy(\R^N), 0\leq \zeta\leq 1$ such that
\begin{equation*}
\zeta =0 \text{  on }\R^N\setminus B_{5R},   \zeta=1 \text{  on }B_{4R},   |\nabla \zeta|\leq \f{C}{R} \text{  and } \left|\f{\pt^2\zeta}{\pt x_i\pt x_j}\right|\leq \f{C}{R^2} \text{  on }\R^N,
\end{equation*}
where the constant $C>0$ is absolute.
Define $\wt \vp_0=\zeta \wh\vp_0$. Then $\wt\vp_0$ is a solution of the following system
\begin{equation*}
\begin{cases}
\pt_t\wt\vp_0+\Div(w \nabla \wt \vp_0)=2w\nabla\zeta\cdot \nabla\wh\vp_0+\wh\vp_0\Div(w\nabla\zeta), &\text{ in }Q,\\
\wt\vp_0=0, &\text{ on }\pt Q, \\
\wt\vp_0(T)=\zeta \vp_T, &\text{ in }\Om.
\end{cases}
\end{equation*}
Let $g=2w\nabla\zeta\cdot \nabla\wh\vp_0+\wh\vp_0\Div(w\nabla\zeta)$. Since $g\in L^2(Q;w^{-1})$ due to $\supp g\subset (B_{5R}\setminus B_{4R})\times (0,T)$, applying Theorem \ref{12.23.T1} to $\wt\vp_0$ yields
\begin{equation*}
\begin{split}
&s\iint_{(\Om\setminus B_R)\times (0,T)} \Theta \psi^\al |\nabla \wt\vp_0|^2e^{2s\xi_0}\df x\df t+s^3\iint_Q\Theta^3\psi^{2-\al}\wt \vp_0^2e^{2s\xi_0}\df x\df t\\
&\leq C\iint_{Q}e^{2s\xi_0}\left(2w\nabla\zeta\cdot\nabla\wh\vp_0+\wh\vp_0\Div(w\nabla\zeta)\right)^2\df x\df t+Cs\iint_{\pt Q}\psi^\al \left(\f{\pt\wt\vp_0}{\pt \nu}\right)^2(x\cdot\nu)e^{2s\xi_0}\df S\df t\\
&\leq  C\f{1}{R^{2-\al}}\iint_{(B_{5R}\setminus B_{4R})\times (0,T)} e^{2s\xi_0}\left(|\nabla\wh\vp_0|^2+|\wh\vp_0|^2\right)\df x\df t
\end{split}
\end{equation*}
given $\f{\pt\wt\vp_0}{\pt\nu}=0$ on $\pt Q$. By the classical Caccioppoli's inequality (see Lemma \ref{12.23.L1} below): There exists a positive constant $C>0$ such that
\begin{equation*}
\iint_{(B_{5R}\setminus B_{4R})\times (0,T)} e^{2s\xi_0}|\nabla \wh\vp_0|^2\df x\df t\leq C\iint_{(B_{6R}\setminus B_{3R})\times (0,T)} (1+\Theta^\f{5}{4})e^{2s\xi_0} \wh\vp_0^2\df x\df t,
\end{equation*}
we obtain
\begin{equation}\label{gbz6}
\begin{split}
&s\iint_{(\Om\setminus B_R)\times (0,T)} \Theta \psi^\al |\nabla \wt\vp_0|^2e^{2s\xi_0}\df x\df t+s^3\iint_Q\Theta^3\psi^{2-\al}\wt \vp_0^2e^{2s\xi_0}\df x\df t\\
&\leq C \iint_{(B_{6R}\setminus B_{3R})\times (0,T)} (1+\Theta^\f{5}{4})e^{2s\xi_0} |\wh\vp_0|^2 \df x\df t,
\end{split}
\end{equation}
where the constant  $C>0$ depends solely on $\al, R, N, T$ and $\Om$.
 We summarize \dref{gbz6} into  Proposition \ref{12.23.P1} following.

 \begin{proposition}\label{12.23.P1}
Let $\widehat{\varphi}_0$ be the solution of equation \eqref{12.15.1} with $g=0$.
Then there exists a constant $C=C(\alpha, R, N, T,\Omega)>0$ such that
\begin{equation*}
\begin{split}
&\iint_{(B_{4R}\setminus B_R)\times (0,T)} \Theta \psi^\alpha |\nabla \widehat{\varphi}_0|^2e^{2s\xi_0}\,dx\,dt+\iint_{B_{4R}\times (0,T)}\Theta^3\psi^{2-\alpha} |\widehat{\varphi}_0|^2e^{2s\xi_0}\,dx\,dt\\
&\leq +C\iint_{(B_{6R}\setminus B_{3R})\times (0,T)} (1+\Theta^{\frac{5}{4}})e^{2s\xi_0}\widehat{\varphi}_0^2\,dx\,dt.
\end{split}
\end{equation*}
\end{proposition}
{\it Step 6}. Using the inequalities
\begin{equation*}
\Theta^3e^{2s\xi_0}\geq C>0   \text{  in }\Omega\times \left(\frac{T}{4}, \frac{3T}{4}\right)
\end{equation*}
and
\begin{equation*}
\Theta^3e^{2s\xi_0}\leq C   \text{  in }Q,
\end{equation*}
where the constant $C>0$ depend only on $\alpha, R, N, T$ and $\Omega$, we have thus established
the following Theorem \ref{12.25.T1}.

\begin{theorem}\label{12.25.T1}
Under the assumptions in Proposition \ref{12.23.P1}, there exists a constant $C=C(\alpha, R, N, T,\Omega)>0$ such that
\begin{equation*}
\begin{split}
&\iint_{(B_{4R}\setminus B_R)\times \left(\frac{T}{4}, \frac{3T}{4}\right)}|x|^\alpha |\nabla \widehat{\varphi}_0|^2\,dx\,dt+\iint_{B_{4R}\times \left(\frac{T}{4}, \frac{3T}{4}\right)} |x|^{2-\alpha}|\widehat{\varphi}_0|^2\,dx\,dt\\
&\leq +C\iint_{(B_{6R}\setminus B_{3R})\times (0,T)} \widehat{\varphi}_0^2\,dx\,dt.
\end{split}
\end{equation*}
\end{theorem}

\begin{remark}
	 For the approximate theorem concerning degenerate elliptic equations and parabolic equations, selecting appropriate approximate functions is crucial. In \cite{Wu1}, we opted for a simple approximate function. Now, we introduce another approximate function that possesses higher-order partial derivatives. These higher-order partial derivatives are pivotal for the Carleman estimate, as the computational process of the Carleman estimate necessitates them.
\end{remark}

\begin{lemma}\label{12.23.L1}
Under the assumptions in Proposition \ref{12.23.P1}, there exists a constant $C>0$ such that
\begin{equation*}
+\iint_{(B_{5R}\setminus B_{4R})\times (0,T)}e^{2s\xi_0}|\nabla \widehat{\varphi}_0|^2\,dx\,dt\leq C\iint_{(B_{6R}\setminus B_{3R})\times (0,T)}(1+
\Theta^{\frac{5}{4}})e^{2s\xi_0}\widehat{\varphi}_0^2\,dx\,dt.
\end{equation*}
\end{lemma}
 \begin{proof}
Select $\zeta \in C^{\infty}(\mathbb{R}^N)$ satisfying $0 \leq \zeta \leq 1$ such that
\begin{equation*}
\zeta = 1 \quad \text{on } B_{5R} \setminus B_{4R}, \quad \zeta = 0 \quad \text{on } (\Omega \setminus B_{6R}) \cup B_{3R}, \quad |\nabla \zeta| \leq \frac{C}{R} \quad \text{on } \mathbb{R}^N.
\end{equation*}
Define $h = e^{2s\xi_0}\zeta^2 \widehat{\varphi}_0$. Then $h(0) = h(T) = 0$, and
\begin{equation*}
\begin{split}
&\iint_Q \zeta^2 e^{2s\xi_0} w |\nabla \widehat{\varphi}_0|^2 \, dx \, dt + 2\iint_Q \zeta e^{2s\xi_0} \widehat{\varphi}_0 (w \nabla \widehat{\varphi}_0 \cdot \nabla \zeta) \, dx \, dt \\
&+ 2s\iint_Q \zeta^2 e^{2s\xi_0} \widehat{\varphi}_0 (w \nabla \widehat{\varphi}_0 \cdot \nabla \xi_0) \, dx \, dt + s\iint_Q \zeta^2 e^{2s\xi_0} \widehat{\varphi}_0^2 (\partial_t \xi_0) \, dx \, dt = -\iint_Q \zeta^2 e^{2s\xi_0} \widehat{\varphi}_0 f \, dx \, dt.
\end{split}
\end{equation*}
This equation is obtained by multiplying $h$ on both sides of \eqref{12.14.1}. Applying the Cauchy inequality, noting $(3R)^{2-\alpha} \leq w \leq (6R)^{2-\alpha}$, and $f = 0$, we obtain
\begin{equation*}
\begin{split}
&\iint_{(B_{5R} \setminus B_{4R}) \times (0,T)} e^{2s\xi_0} |\nabla \widehat{\varphi}_0|^2 \, dx \, dt \\
&\leq C\iint_Q \zeta^2 e^{2s\xi_0} w |\nabla \widehat{\varphi}_0|^2 \, dx \, dt \leq C\iint_{(B_{6R} \setminus B_{3R}) \times (0,T)} (1 + \Theta^{\frac{5}{4}}) e^{2s\xi_0} \widehat{\varphi}_0^2 \, dx \, dt.
\end{split}
\end{equation*}
This completes the proof of Lemma \ref{12.23.L1}.
\end{proof}

\section{Quantitative weak unique continuation}\label{Section 6}

  In this section, we begin by presenting Theorem \ref{04.18.T1g}. This theorem employs a cut-off method: by isolating the degenerate part, the remaining portion of equation \eqref{12.14.1} transforms into a non-degenerate parabolic equation, for which the classical Carleman estimate can then be established. Although the cut-off method has been previously demonstrated in \cite{Wu}, we include its proof in Section \ref{Section 5} as Appendix for the sake of completeness. It is important to note, however, that the effectiveness of the cut-off method hinges on the selection of an appropriate cut-off function. Finally, we introduce the principal result of this paper, Theorem \ref{04.18.T1-final}.

 \begin{theorem}\label{04.18.T1g}
 	Let $\widehat{\varphi}_0$ be a solution of \eqref{12.15.1} with $g = 0$. Then, there exist constants $C = C(\alpha, R, N, T, \Omega) > 0$ such that
 	\begin{equation*}
 		\begin{split}
 			&\iint_{(\Omega \setminus B_{5R}) \times \left(\frac{T}{4}, \frac{3T}{4}\right)} |x|^\alpha |\nabla\widehat{\varphi}_0|^2 \, dx \, dt + \iint_{(\Omega \setminus B_{5R}) \times \left(\frac{T}{4}, \frac{3T}{4}\right)} |x|^{2-\alpha} |\widehat{\varphi}_0|^2 \, dx \, dt\\
 			&\leq C \iint_{(B_{6R} \setminus B_{3R}) \times (0,T)} \widehat{\varphi}_0^2 \, dx \, dt.
 		\end{split}
 	\end{equation*}
 \end{theorem}

By integrating $\omega = A_{3R,6R}$ with Theorem \ref{12.25.T1} and Theorem \ref{04.18.T1g}, we deduce the following Theorem \ref{03.01.T1}.

\begin{theorem}\label{03.01.T1}
Let $\widehat{\varphi}_0$ be a solution of \eqref{12.15.1} with $g = 0$. Then, there exists a constant $C = C(\alpha, R, N, T, \Omega) > 0$ such that
\begin{equation*}
\iint_{\Omega \times \left(\frac{T}{4}, \frac{3T}{4}\right)} |x|^{2-\alpha} |\widehat{\varphi}_0|^2 \, dx \, dt \leq C \iint_{\omega \times (0,T)} |\widehat{\varphi}_0|^2 \, dx \, dt.
\end{equation*}
\end{theorem}
It is noted that the aforementioned results rely on the terminal data $u_T \in H_0^1(\Omega) \cap H^3(\Omega)$. For each $u_T \in L^2(\Omega)$, since $C_0^\infty(\Omega)$ is dense in $L^2(\Omega)$, there exists a sequence $\{u_T^n\}_{n \in \mathbb{N}} \in H_0^1(\Omega) \cap H^3(\Omega)$ with $\Div(w_k \nabla u_T^n) \in H_0^1(\Omega)$ for all $k \in \mathbb{N}$ (refer to Part 2 in Corollary \ref{08.23.C1}), such that
\begin{equation*}
u_T^n \to u_T   \text{ in } L^2(\Omega).
\end{equation*}
Let $\widehat{\varphi}_n$ be the solution of \eqref{12.15.1} with $\widehat{\varphi}_n(T) = u_T^n$ and $g = 0$, and $\widehat{\varphi}_0$ be the solution of \eqref{12.15.1} with $\widehat{\varphi}_0(T) = u_T \in L^2(\Omega)$. Denote $\widehat{\psi}_n = \widehat{\varphi}_n - \widehat{\varphi}_0$. Then, $\widehat{\psi}_n$ is the solution of the following equation:
\begin{equation*}
\begin{cases}
\partial_t \widehat{\psi}_n + \Div(\nabla \widehat{\psi}_n) = 0, &\text{ in } Q, \\
\widehat{\psi}_n = 0, &\text{ on } \partial Q, \\
\widehat{\psi}_n(T) = \widehat{\varphi}_n(T) - \widehat{\varphi}_0(T), &\text{ in } \Omega,
\end{cases}
\end{equation*}
and we have the following estimate:
\begin{equation*}
\|\widehat{\psi}_n\|{L^2(Q)} \leq \sqrt{T} \|u_T^n - u_T\|_{L^2(\Omega)} \to 0   \text{ as } n \to \infty
\end{equation*}
by a straightforward computation and Corollary \ref{12.11.C1}.
Since $|x|^{2-\alpha}$ is a bounded continuous function on $\Omega$, by applying Theorem \ref{03.01.T1} once more, we obtain the following Theorem \ref{04.18.T1-2}.
\begin{theorem}\label{04.18.T1-2}
Let $\widehat{\varphi}_0$ be a solution of \eqref{12.14.1}. Then, there exists a constant $C = C(\alpha, R, N, T, \Omega) > 0$ such that
\begin{equation}\label{guo1} 
\iint_{\Omega \times \left(\frac{T}{4}, \frac{3T}{4}\right)} |x|^{2-\alpha} |\widehat{\varphi}_0|^2 \, dx \, dt \leq C \iint_{\omega \times (0,T)} |\widehat{\varphi}_0|^2 \, dx \, dt.
\end{equation}
\end{theorem}
 From Theorem \ref{04.18.T1-2}, we can directly deduce a \emph{Quantitative} weak unique continuation property: should $\widehat{\varphi}_0$ vanish on $\omega \times (0,T)$, then,
\begin{equation*}
\iint_{\Omega \times \left(\frac{T}{4}, \frac{3T}{4}\right)} |x|^{2-\alpha} |\widehat{\varphi}_0|^2 \, dx \, dt = 0,
\end{equation*}
i.e., $\widehat{\varphi}_0 = 0$ on $\Omega \times \left(\frac{T}{4}, \frac{3T}{4}\right)$. By Corollary \ref{12.11.C1}, we obtain
\begin{equation}\label{guo2} 
\int_\Omega |\widehat{\varphi}_0|^2(0) \, dx \leq \frac{2}{T} \iint_{\Omega \times \left(\frac{T}{4}, \frac{3T}{4}\right)} |\widehat{\varphi}_0|^2(t) \, dx \, dt = 0,
\end{equation}
i.e., $\widehat{\varphi}_0(0) = 0$ on $\Omega$. This leads to the following Corollary \ref{04.18.T1-final}. 

\begin{corollary}\label{04.18.T1-final}
Let $\widehat{\varphi}_0$ be a solution of \eqref{12.14.1}. Then, the solution $\widehat{\varphi}_0$ exhibits the weak unique continuation property. That is, if $\widehat{\varphi}_0 = 0$ on $\omega \times (0,T)$, then $\widehat{\varphi}_0(0) = 0$ on $\Omega$.
\end{corollary}

\section{Appendix: Carleman estimate for the cut-off parabolic equation}\label{Section 5}

In this section, we begin by performing some computations, with the objective of establishing Theorem \ref{04.18.T1g}. We note that while the Carleman estimate for parabolic equations in this section is standard, a straightforward yet effective cut-off method (see \cite{Wu}) is required. Under an appropriate cut-off, degenerate elliptic and parabolic equations exhibit unique continuation properties; moreover, controllability holds for degenerate parabolic equations. Consequently, we rederive this estimate.

We select $\kappa \in C^\infty(\mathbb{R}^N)$, where $0 \leq \kappa \leq 1$, such that
\begin{equation*}
	\kappa = 1 \text{  on } \mathbb{R}^N \setminus B_{5R},   \kappa = \nabla \kappa = 0 \text{  on } B_{4R},   |\nabla \kappa| \leq \frac{C}{R} \text{  and } \left|\frac{\partial^2 \kappa}{\partial x_i \partial x_j}\right| \leq \frac{C}{R^2} \text{  on } \mathbb{R}^N,
\end{equation*}
and noting that $\widehat{\varphi}_0$ is the solution of \eqref{12.15.1} with $g=0$, then $\overline{\varphi}_0 = \kappa \widehat{\varphi}_0$ is a solution of the following system:
\begin{equation}\label{12.24.1}
	\begin{cases}
		\partial_t \overline{\varphi}_0 + \text{ Div}(w \nabla \overline{\varphi}_0) = 2w \nabla \kappa \cdot \nabla \widehat{\varphi}_0 + \widehat{\varphi}_0 \text{ Div}(w \nabla \kappa), &\text{ in } (\Omega \setminus B_{4R}) \times (0,T),\\
		\overline{\varphi}_0 = 0, &\text{ on } (\partial \Omega \cup \partial B{4R}) \times (0,T),\\
		\overline{\varphi}_0(T) = \kappa \varphi_T, &\text{ in } \Omega \setminus B_{4R},
	\end{cases}
\end{equation}
which is a uniformly parabolic equation. Denote $g = 2w \nabla \kappa \cdot \nabla \widehat{\varphi}_0 + \widehat{\varphi}_0 \text{ Div}(w \nabla \kappa)$.
\begin{lemma}(\cite[lemma 1.2]{FG})\label{12.23.L2}
	There exists $\overline{\eta} \in C^2(\overline{\Omega})$ such that $\overline{\eta} > 0$ in $\Omega$, $\overline{\eta} = 0$ on $\partial \Omega \cup \partial B_{4R}$, and $|\nabla \overline{\eta}| > 0$ in $\overline{\Omega \setminus B_{5R}}$.
\end{lemma}
Let $\Theta(t) = \frac{1}{[t(T-t)]^4}$, and define $|\overline{\eta}|\infty = \text{esssup}{x \in \Omega} \overline{\eta}(x)$. Then, we have
$$
\overline{\xi}(x,t) = \Theta(t) e^{\lambda(8|\overline{\eta}|_\infty + \overline{\eta}(x))}, \quad \overline{\sigma}(x,t) = \Theta(t) e^{10\lambda|\overline{\eta}|_\infty} - \overline{\xi}(x,t).
$$
Let $\overline{Q} = (\Omega \setminus B_R) \times (0,T)$, and
$$
\overline{v} = e^{-s\overline{\sigma}} \overline{\varphi}_0 \quad \iff \quad \overline{\varphi}_0 = e^{s\overline{\sigma}} \overline{v}.
$$
By direct computation, we obtain
$$
\begin{aligned}
	e^{-s\overline{\sigma}}g
	&= \partial_t \overline{v} + 2s w \nabla \overline{v} \cdot \nabla \overline{\sigma} + s \overline{v} \text{Div}(w \nabla \overline{\sigma}) \\
	&\quad + \text{Div}(w \nabla \overline{v}) + s \overline{v} \overline{\sigma}_t + s^2 \overline{v} w \nabla \overline{\sigma} \cdot \nabla \overline{\sigma}.
\end{aligned}
$$
This implies:
$$
\overline{g} = \overline{P}_1 \overline{v} + \overline{P}_2 \overline{v},
$$
where
$$
\begin{aligned}
	\overline{g} := \sum_{i=1}^3 \overline{g}_i
	&= e^{-s\overline{\sigma}}g + s\lambda \overline{\xi} \overline{v} \text{Div}(w \nabla \overline{\eta}) - s\lambda^2 \overline{\xi} \overline{v} w \nabla \overline{\eta} \cdot \nabla \overline{\eta}, \\
	\overline{P}_1 \overline{v} := \sum_{i=1}^3 P_{1i} \overline{v}
	&= \partial_t \overline{v} - 2s\lambda \overline{\xi} w \nabla \overline{v} \cdot \nabla \overline{\eta} - 2s\lambda^2 \overline{\xi} \overline{v} w \nabla \overline{\eta} \cdot \nabla \overline{\eta}, \\
	\overline{P}_2 \overline{v} := \sum_{i=1}^3 \overline{P}_{2i} \overline{v}
	&= \text{Div}(w \nabla \overline{v}) + s \overline{v} \overline{\sigma}_t + s^2 \lambda^2 \overline{\xi}^2 \overline{v} w \nabla \overline{\eta} \cdot \nabla \overline{\eta}.
\end{aligned}
$$
These results follow from
$$
\nabla \overline{\sigma} = -\nabla \overline{\xi} = -\lambda \overline{\xi} \nabla \overline{\eta}, \quad \text{Div}(w \nabla \overline{\sigma}) = -\lambda \overline{\xi} \text{Div}(w \nabla \overline{\eta}) - \lambda^2 \overline{\xi} w \nabla \overline{\eta} \cdot \nabla \overline{\eta}.
$$

It can be readily confirmed that:
a) $\overline{v}$ and its first-order partial derivative with respect to $x_i$, denoted as $\frac{\partial \overline{v}}{\partial x_i}$, both vanish in $L^2(\Omega \setminus B_R)$ at the time instances $t=0$ and $t=T$;
b) $\overline{v}$ equals zero on the boundary $\partial \overline{Q}$, which is defined as $[\partial \Omega \times (0,T)] \cup [\partial B_R \times (0,T)]$.

Next, we proceed to compute and estimate the expression $(\overline{P}_1 \overline{v}, \overline{P}_2 \overline{v})_{L^2(\overline{Q})}$ by examining each term individually.

{\it (i): Calculate $(\overline{P}_{11} \overline{v}, \overline{P}_{21} \overline{v})_{L^2(\overline{Q})}$.}

Indeed, from conditions a) and b), we obtain
\begin{equation*}
	\begin{split}
		(\overline{P}_{11} \overline{v}, \overline{P}_{21} \overline{v})_{L^2(\overline{Q})}
		&= \iint_{\overline{Q}} (\partial_t \overline{v}) \text{Div}(w \nabla \overline{v}) \, dx \, dt = -\iint_{\overline{Q}} w \nabla \overline{v} \cdot \nabla \partial_t \overline{v} \, dx \, dt\\
		&= -\frac{1}{2} \iint_{\overline{Q}} \partial_t \left(w |\nabla \overline{v}|^2\right) \, dx \, dt = 0.
	\end{split}
\end{equation*}

{\it (ii): compute and estimate $(\overline{P}_{12} \overline{v}, \overline{P}_{21} \overline{v})_{L^2(\overline{Q})}$.}

Given that $\frac{\partial \overline{\varphi}_0}{\partial \nu} = (\nabla \kappa \cdot \nu) \widehat{\varphi}_0 + \kappa (\nabla \widehat{\varphi}_0 \cdot \nu) = 0$ on $\partial B_R$, we have
\begin{eqnarray*}
	&&(\overline{P}_{12} \overline{v}, \overline{P}_{21} \overline{v})_{L^2(\overline{Q})}\\
	&&= -2s\lambda \iint_{\overline{Q}} \overline{\xi} (w \nabla \overline{v} \cdot \nabla \overline{\eta}) \text{Div}(w \nabla \overline{v}) \, dx \, dt\\
	&&= -2s\lambda \iint_{\overline{Q}} \text{Div} \left[\overline{\xi} (w \nabla \overline{v} \cdot \nabla \overline{\eta}) w \nabla \overline{v}\right] \, dx \, dt + 2s\lambda \iint_{\overline{Q}} w \nabla \overline{v} \cdot \nabla \left[\overline{\xi} (w \nabla \overline{v} \cdot \nabla \overline{\eta})\right] \, dx \, dt\\
	&&= -2s\lambda \iint_{\partial \Omega \times (0,T)} \overline{\xi} (w \nabla \overline{v} \cdot \nabla \overline{\eta}) w \nabla \overline{v} \cdot \nu \, dS \, dt + 2s\lambda \iint_{\partial B_R \times (0,T)} \overline{\xi} (w \nabla \overline{v} \cdot \nabla \overline{\eta}) w \nabla \overline{v} \cdot \nu \, dS \, dt\\
	&&\hspace{4.5mm} + 2s\lambda \iint_{\overline{Q}} (w \nabla \overline{v} \cdot \nabla \overline{\xi}) (w \nabla \overline{v} \cdot \nabla \overline{\eta}) \, dx \, dt + 2s\lambda \iint_{\overline{Q}} \overline{\xi} (w \nabla \overline{v} \cdot \nabla w) (\nabla \overline{v} \cdot \nabla \overline{\eta}) \, dx \, dt\\
	&&\hspace{4.5mm} + 2s\lambda \iint_{\overline{Q}} \overline{\xi} w^2 \nabla \overline{v} \cdot \nabla (\nabla \overline{v} \cdot \nabla \overline{\eta}) \, dx \, dt\\
	&&= s\lambda \iint_{\partial \Omega \times (0,T)} \overline{\xi} |x|^{2\alpha} \left|\frac{\partial \overline{v}}{\partial \nu}\right|^2 \left|\frac{\partial \overline{\eta}}{\partial \nu}\right| \, dS \, dt + 2s\lambda \iint_{\overline{Q}} \overline{\xi} |x|^{2\alpha} (D^2 \overline{\eta} \nabla \overline{v}) \cdot \nabla \overline{v} \, dx \, dt\\
	&&\hspace{4.5mm} + 2s\lambda^2 \iint_{\overline{Q}} \overline{\xi} |x|^{2\alpha} (\nabla \overline{v} \cdot \nabla \overline{\eta})^2 \, dx \, dt + 2\alpha s\lambda \iint_{\overline{Q}} \overline{\xi} |x|^{2\alpha-2} (\nabla \overline{v} \cdot x) (\nabla \overline{v} \cdot \nabla \overline{\eta}) \, dx \, dt\\
	&&\hspace{4.5mm} - \alpha s\lambda \iint_{\overline{Q}} \overline{\xi} |x|^{2\alpha-2} |\nabla \overline{v}|^2 (x \cdot \nabla \overline{\eta}) \, dx \, dt - s\lambda \iint_{\overline{Q}} \overline{\xi} |x|^\alpha |\nabla \overline{v}|^2 \text{Div}(|x|^\alpha \nabla \overline{\eta}) \, dx \, dt\\
	&&\hspace{4.5mm} - s\lambda^2 \iint_{\overline{Q}} \overline{\xi} |x|^{2\alpha} |\nabla \overline{v}|^2 |\nabla \overline{\eta}|^2 \, dx \, dt
\end{eqnarray*}
The following equalities are used in the above derivation
\begin{equation*}
	2s\lambda \iint_{\overline{Q}} (w \nabla \overline{v} \cdot \nabla \overline{\xi}) (w \nabla \overline{v} \cdot \nabla \overline{\eta}) \, dx \, dt = 2s\lambda^2 \iint_{\overline{Q}} \overline{\xi} |x|^{2\alpha} (\nabla \overline{v} \cdot \nabla \overline{\eta})^2 \, dx \, dt,
\end{equation*}
and
\begin{equation*}
	\begin{split}
		2s\lambda \iint_{\overline{Q}} \overline{\xi} (w \nabla \overline{v} \cdot \nabla w) (\nabla \overline{v} \cdot \nabla \overline{\eta}) \, dx \, dt = 2\alpha s\lambda \iint_{\overline{Q}} \overline{\xi} |x|^{2\alpha-2} (\nabla \overline{v} \cdot x) (\nabla \overline{v} \cdot \nabla \overline{\eta}) \, dx \, dt,
	\end{split}
\end{equation*}
and by $\nabla \overline{v} \cdot \nabla (\nabla \overline{v} \cdot \nabla \overline{\eta}) = \frac{1}{2} \nabla \overline{\eta} \cdot \nabla |\nabla \overline{v}|^2 + (D^2 \overline{\eta} \nabla \overline{v}) \cdot \nabla \overline{v}$, we have
\begin{eqnarray*}
	&&2s\lambda \iint_{\overline{Q}} \overline{\xi} w^2 \nabla \overline{v} \cdot \nabla (\nabla \overline{v} \cdot \nabla \overline{\eta}) \, dx \, dt\\
	&&= s\lambda \iint_{\overline{Q}} \overline{\xi} |x|^{2\alpha} \nabla \overline{\eta} \cdot \nabla |\nabla \overline{v}|^2 \, dx \, dt + 2s\lambda \iint_{\overline{Q}} \overline{\xi} |x|^{2\alpha} (D^2 \overline{\eta} \nabla \overline{v}) \cdot \nabla \overline{v} \, dx \, dt\\
	&&= -s\lambda \iint_{\partial \Omega \times (0,T)} \overline{\xi} |x|^{2\alpha} \left|\frac{\nabla \overline{v}}{\partial \nu}\right|^2 \left|\frac{\partial \overline{\eta}}{\partial \nu}\right| \, dS \, dt + s\lambda \iint_{\partial B_R \times (0,T)} \overline{\xi} |x|^{2\alpha} \left|\frac{\nabla \overline{v}}{\partial \nu}\right|^2 \left|\frac{\partial \overline{\eta}}{\partial \nu}\right| \, dS \, dt\\
	&&\hspace{4.5mm} - s\lambda \iint_{\overline{Q}} |\nabla \overline{v}|^2 \text{Div} \left(\overline{\xi} |x|^{2\alpha} \nabla \overline{\eta}\right) \, dx \, dt + 2s\lambda \iint_{\overline{Q}} \overline{\xi} |x|^{2\alpha} (D^2 \overline{\eta} \nabla \overline{v}) \cdot \nabla \overline{v} \, dx \, dt\\
	&&= -s\lambda \iint_{\partial \Omega \times (0,T)} \overline{\xi} |x|^{2\alpha} \left|\frac{\nabla \overline{v}}{\partial \nu}\right|^2 \left|\frac{\partial \overline{\eta}}{\partial \nu}\right| \, dS \, dt + 2s\lambda \iint_{\overline{Q}} \overline{\xi} |x|^{2\alpha} (D^2 \overline{\eta} \nabla \overline{v}) \cdot \nabla \overline{v} \, dx \, dt\\
	&&\hspace{4.5mm} - s\lambda^2 \iint_{\overline{Q}} \overline{\xi} |x|^{2\alpha} |\nabla \overline{v}|^2 |\nabla \overline{\eta}|^2 \, dx \, dt - \alpha s\lambda \iint_{\overline{Q}} \overline{\xi} |x|^{2\alpha-2} |\nabla \overline{v}|^2 (x \cdot \nabla \overline{\eta}) \, dx \, dt\\
	&&\hspace{4.5mm} - s\lambda \iint_{\overline{Q}} \overline{\xi} |x|^\alpha |\nabla \overline{v}|^2 \text{Div}(|x|^\alpha \nabla \overline{\eta}) \, dx \, dt.
\end{eqnarray*}

{\it (iii): Compute $(\overline{P}_{13}\overline{v}, \overline{P}_{21}\overline{v})_{L^2(\overline{Q})}$.}

Indeed, given $\overline{v}=0$ on $\partial\overline{Q}$, we have
\begin{eqnarray*}
	&&(\overline{P}_{13}\overline{v}, \overline{P}_{21}\overline{v})_{L^2(\overline{Q})}\\
	&&=-2s\lambda^2\iint_{\overline{Q}}\overline{\xi}\overline{v}w|\nabla\overline{\eta}|^2 \Div(w\nabla \overline{v})\,dx\,dt\\
	&&=+2s\lambda^2\iint_{\overline{Q}}w\nabla \overline{v} \cdot \nabla \left[\overline{\xi} \overline{v}w|\nabla\overline{\eta}|^2\right]\,dx\,dt\\
	&&=+2s\lambda^2\iint_{\overline{Q}}\overline{\xi} |x|^{2\alpha}|\nabla \overline{v}|^2|\nabla \overline{\eta}|^2\,dx\,dt + 2s\lambda^3\iint_{\overline{Q}}\overline{\xi} |x|^{2\alpha} \overline{v}(\nabla \overline{v} \cdot \nabla \overline{\eta})|\nabla \overline{\eta}|^2\,dx\,dt\\
	&&\quad + 2\alpha s\lambda^2\iint_{\overline{Q}}\overline{\xi}|x|^{2\alpha-2}\overline{v}(\nabla \overline{v} \cdot x)|\nabla \overline{\eta}|^2\,dx\,dt + 4s\lambda^2\iint_{\overline{Q}}\overline{\xi} |x|^{2\alpha}\overline{v}(D^2\overline{\eta}\nabla \overline{\eta}) \cdot \nabla \overline{v}\,dx\,dt.
\end{eqnarray*}

The (i)-(iii), and  $R\leq |x|\leq m$ and $\eta\in C^2(\overline{\Omega})$, and
\begin{eqnarray*}
	&& 2s\lambda\iint_{\overline{Q}}\overline{\xi}|x|^{2\alpha}(D^2\overline{\eta}\nabla \overline{v}) \cdot \nabla \overline{v}\,dx\,dt \geq -Cs\lambda\iint_{\overline{Q}}\overline{\xi} |\nabla \overline{v}|^2\,dx\,dt, \\
	&&2\alpha s\lambda\iint_{\overline{Q}}\overline{\xi}|x|^{2\alpha-2}(\nabla \overline{v} \cdot x)(\nabla \overline{v} \cdot \nabla \overline{\eta})\,dx\,dt \geq -Cs\lambda\iint_{\overline{Q}}\overline{\xi} |\nabla \overline{v}|^2\,dx\,dt,\\
	&&-\alpha s\lambda\iint_{\overline{Q}}\overline{\xi} |x|^{2\alpha-2}|\nabla \overline{v}|^2(x \cdot \nabla \overline{\eta})\,dx\,dt \geq -Cs\lambda \iint_{\overline{Q}}\overline{\xi} |\nabla \overline{v}|^2\,dx\,dt, \\
	&&-s\lambda\iint_{\overline{Q}}\overline{\xi} |x|^\alpha |\nabla \overline{v}|^2\Div(|x|^\alpha \nabla \eta)\,dx\,dt \geq -Cs\lambda \iint_{\overline{Q}}\overline{\xi}|\nabla \overline{v}|^2\,dx\,dt,
\end{eqnarray*}
and
\begin{equation*}
	\begin{split}
		&2s\lambda^3\iint_{\overline{Q}}\overline{\xi} |x|^{2\alpha}\overline{v}(\nabla \overline{v} \cdot \nabla \overline{\eta})|\nabla \overline{\eta}|^2\,dx\,dt\\
		&\geq -C \lambda^2\iint_{\overline{Q}}\overline{\xi} (\nabla \overline{v} \cdot \nabla \overline{\eta})^2\,dx\,dt - Cs^2\lambda^4\iint_{\overline{Q}} \overline{\xi} \overline{v}^2|\nabla \overline{\eta}|^4\,dx\,dt,
	\end{split}
\end{equation*}
and
\begin{equation*}
	\begin{split}
		&2\alpha s\lambda^2\iint_{\overline{Q}}\overline{\xi} |x|^{2\alpha-2}\overline{v}(\nabla\overline{v} \cdot x)|\nabla \overline{\eta}|^2\,dx\,dt\\
		&\geq -Cs\lambda\iint_{\overline{Q}}\overline{\xi} |\nabla \overline{v}|^2\,dx\,dt - Cs\lambda^3\iint_{\overline{Q}}\overline{\xi} \overline{v}^2|\nabla\overline{\eta}|^4\,dx\,dt,
	\end{split}
\end{equation*}
and
\begin{equation*}
	\begin{split}
		&+4s\lambda^2\iint_{\overline{Q}}\overline{\xi} |x|^{2\alpha}\overline{v}(D^2\overline{\eta}\nabla \overline{\eta})\nabla \overline{v}\,dx\,dt \geq -Cs\iint_{\overline{Q}}\overline{\xi} |\nabla \overline{v}|^2\,dx\,dt - Cs\lambda^4\iint_{\overline{Q}} \overline{\xi} \overline{v}^2\,dx\,dt,
	\end{split}
\end{equation*}
lead to
\begin{equation}\label{12.24.2}
	\begin{split}
		&(\overline{P}_{12}\overline{v}, \overline{P}_{21}\overline{v}){L^2(\overline{Q})}+(\overline{P}_{13}, \overline{P}_{21}){L^2(\overline{Q})}\\
		&\geq +s\lambda\iint_{\partial\Omega\times (0,T)}\overline{\xi}|x|^{2\alpha}\left|\frac{\partial \overline{v}}{\partial \nu}\right|^2\left|\frac{\partial \overline{\eta}}{\partial \nu}\right|\,dS\,dt + R^{2\alpha}s\lambda^2\iint_{\overline{Q}}\overline{\xi}(\nabla \overline{v} \cdot\nabla \overline{\eta})^2\,dx\,dt\\
		&\quad + 2R^{2\alpha}s\lambda^2\iint_{\overline{Q}}\overline{\xi} |\nabla \overline{v}|^2|\nabla \overline{\eta}|^2\,dx\,dt\\
		&\quad - Cs\lambda\iint_{\overline{Q}}\overline{\xi} |\nabla \overline{v}|^2\,dx\,dt - Cs^2\lambda^4\iint_{\overline{Q}}\overline{\xi} \overline{v}^2|\nabla \overline{\eta}|^4\,dx\,dt - Cs\lambda^4\iint_{\overline{Q}}\overline{\xi} \overline{v}^2\,dx\,dt
	\end{split}
\end{equation}
for $\lambda\geq C\geq 1$ and $s\geq C\geq 1$.

{\it Compute $(\overline{P}_{11}\overline{v}, \overline{P}_{22} \overline{v})_{L^2(\overline{Q})}$.}

Indeed, by a), we have
\begin{equation*}
	\begin{split}
		&(\overline{P}_{11}\overline{v}, \overline{P}_{22} \overline{v}){L^2(\overline{Q})}=+s \iint{\overline{Q}}(\partial_t\overline{v})\overline{v}\overline{\sigma}_t\,dx\,dt=-\frac{s}{2}\iint_{\overline{Q}}\overline{v}^2\overline{\sigma}_{tt}\,dx\,dt.
	\end{split}
\end{equation*}

{\it (v): Compute $(\overline{P}_{12}\overline{v}, \overline{P}_{22}\overline{v})_{L^2(\overline{Q})}$.}

Indeed, by b), we have
\begin{equation*}
	\begin{split}
		&(\overline{P}_{12}\overline{v}, \overline{P}_{22}\overline{v})_{L^2(\overline{Q})}\\
		&=-2s^2\lambda \iint_{\overline{Q}}\overline{\xi} w(\nabla \overline{v}\cdot \nabla \overline{\eta}) \overline{v} \overline{\sigma}_t\,dx\,dt =+s^2\lambda\iint_{\overline{Q}}\overline{v}^2\Div\left[\overline{\xi} \overline{\sigma}_t w\nabla \overline{\eta}\right]\,dx\,dt\\
		&=+s^2\lambda^2\iint_{\overline{Q}}\overline{\xi} \overline{\sigma}_t \overline{v}^2|x|^\alpha  |\nabla \overline{\eta}|^2 \,dx\,dt+s^2\lambda\iint_{\overline{Q}}\overline{\xi} \overline{\sigma}_t\overline{v}^2\Div(|x|^\alpha \nabla \overline{\eta})\,dx\,dt\\
		&\hspace{4.5mm}-s^2\lambda^2\iint_{\overline{Q}}\overline{\xi}\overline{\xi}_t \overline{v}^2|x|^\alpha |\nabla \overline{\eta}|^2\,dx\,dt.
	\end{split}
\end{equation*}

{\it (vi): Compute $(\overline{P}_{13}\overline{v}, \overline{P}_{22}\overline{v})_{L^2(\overline{Q})}$.}

Indeed, we have
\begin{equation*}
	\begin{split}
		(\overline{P}_{13}\overline{v}, \overline{P}_{22}\overline{v})_{L^2(\overline{Q})}
		&=-2s^2\lambda^2\iint_{\overline{Q}}\overline{\xi} \overline{\sigma}_t\overline{v}^2|x|^\alpha |\nabla \overline{\eta}|^2\,dx\,dt.
	\end{split}
\end{equation*}

{\it (vii): Compute $(\overline{P}_{11}\overline{v}, \overline{P}_{23}\overline{v})_{L^2(\overline{Q})}$.}

Indeed, by i),  we have
\begin{equation*}
	\begin{split}
		(\overline{P}_{11}\overline{v}, \overline{P}_{23}\overline{v})_{L^2(\overline{Q})}
		&=+s^2\lambda^2\iint_{\overline{Q}}\overline{\xi}^2 w|\nabla \overline{\eta}|^2 \overline{v}\partial_t\overline{v}\,dx\,dt=-s^2\lambda^2\iint_{\overline{Q}}\xi\xi_t\overline{v}^2w|\nabla\overline{\eta}|^2\,dx\,dt.
	\end{split}
\end{equation*}
(E2) From (4)-(7), $|\overline{\sigma}_{tt}|\leq C\overline{\xi}^{\frac{3}{2}}\leq C\overline{\xi}^3$, $|\overline{\xi}\overline{\sigma}_t|\leq C\overline{\xi}^3$, $|\xi\xi_t|\leq C\overline{\xi}^3$, we obtain
\begin{equation}\label{12.24.3}
	\begin{split}
		&(\overline{P}_{11}\overline{v},\overline{P}_{22}\overline{v}){L^2(\overline{Q})}+(\overline{P}_{11}\overline{v}, \overline{P}_{23}\overline{v}){L^2(\overline{Q})} \geq
		-Cs^2\lambda^2\iint_{\overline{Q}}\overline{\xi}^3 \overline{v}^2\,dx\,dt.
	\end{split}
\end{equation}

{\it (viii):  Compute $(\overline{P}_{12}\overline{v}, \overline{P}_{23}\overline{v})_{L^2(\overline{Q})}$.}

Indeed, by b), we have
\begin{equation*}
	\begin{split}
		(\overline{P}_{12}\overline{v}, \overline{P}_{23}\overline{v})_{L^2(\overline{Q})}
		&=-2s^3\lambda^3\iint_{\overline{Q}}\overline{\xi}^3 w|\nabla \overline{\eta}|^2 (w\nabla \overline{v}\cdot \nabla \eta)\overline{v}\,dx\,dt\\
		&=+3s^3\lambda^4\iint_{\overline{Q}}\overline{\xi}^3\overline{v}^2|x|^{2\alpha}|\nabla \overline{\eta}|^4\,dx\,dt+s^3\lambda^3\iint_{\overline{Q}}\overline{\xi}^3\overline{v}^2\Div(|x|^{2\alpha}|\nabla \overline{\eta}|^2\nabla \overline{\eta})\,dx\,dt.
	\end{split}
\end{equation*}

{\it (ix): Compute $(\overline{P}_{13}\overline{v}, \overline{P}_{23}\overline{v})_{L^2(\overline{Q})}$.}

Indeed, we have
\begin{equation*}
	\begin{split}
		(\overline{P}_{13}\overline{v}, \overline{P}_{23}\overline{v}){L^2(\overline{Q})}=-2s^3\lambda^4\iint{\overline{Q}}\overline{\xi}^3 \overline{v}^2 |x|^{2\alpha}|\nabla \overline{\eta}|^4\,dx\,dt.
	\end{split}
\end{equation*}

From (viii)-(ix), we obtain
\begin{equation}\label{12.24.4}
	\begin{split}
		(\overline{P}_{12}\overline{v}, \overline{P}_{23}\overline{v}){L^2(\overline{Q})}+(\overline{P}_{13}\overline{v}, \overline{P}_{23}\overline{v}){L^2(\overline{Q})}\geq+ s^3\lambda^4\iint_{\overline{Q}}\overline{\xi}^3\overline{v}^2|\nabla\overline{\eta}|^4\,dx\,dt-Cs^3\lambda^3\iint_{\overline{Q}}\overline{\xi}^3\overline{v}^2\,dx\,dt.
	\end{split}
\end{equation}

{\it (x): Combining \eqref{12.24.1}, \eqref{12.24.2} and \eqref{12.24.3}, by $2(\overline{P}_1\overline{v}, \overline{P}_2\overline{v})_{L^2(\overline{Q})}$,}  we obtain
\begin{equation*}
	\begin{split}
		&+s\lambda^2\iint_{\overline{Q}}\overline{\xi} (\nabla\overline{v}\cdot \nabla \overline{\eta})^2\,dx\,dt+s\lambda^2\iint_{\overline{Q}}\overline{\xi} |\nabla \overline{v}|^2|\nabla \overline{\eta}|^2\,dx\,dt +s^3\lambda^4\iint_{(\Omega\setminus B_{5R})\times (0,T)}\overline{\xi}^3\overline{v}^2|\nabla\overline{\eta}|^4\,dx\,dt\\
		&\leq +Cs^3\lambda^4\iint_{(B_{5R}\setminus B_{4R})\times (0,T)}\overline{\xi} ^3\overline{v}^2\,dx\,dt+Cs\lambda \iint_{\overline{Q}}\overline{\xi} |\nabla \overline{v}|^2\,dx\,dt+C\left\|e^{-s\overline{\sigma}}g\right\|_{L^2(\overline{Q})}.
	\end{split}
\end{equation*}
From Lemma \ref{12.23.L2} (i.e., $|\nabla \overline{\eta}|\geq C>0$ on $\overline{(\Omega\setminus B_R)\setminus \widehat{\omega}}$), we have
\begin{equation}\label{12.24.5}
	\begin{split}
		&+s\lambda^2\iint_{\overline{Q}}\overline{\xi} |\nabla \overline{v}|^2 \,dx\,dt +s^3\lambda^4\iint_{\overline{Q}}\overline{\xi}^3\overline{v}^2 \,dx\,dt\\
		&\leq +Cs^3\lambda^4\iint_{\widehat{\omega} \times (0,T)}\overline{\xi} ^3\overline{v}^2\,dx\,dt+Cs\lambda \iint_{\widehat{\omega} \times (0,T)}\overline{\xi} |\nabla \overline{v}|^2\,dx\,dt+C\left\|e^{-s\overline{\sigma}}g\right\|_{L^2(\overline{Q})}.
	\end{split}
\end{equation}

{\it (xi): It is easily verified that} (from the definition of $\overline{P}_1\overline{v}, \overline{P}_2\overline{v}$)
\begin{equation*}
	\begin{split}
		+s^{-1}\iint_{\overline{Q}}\overline{\xi}^{-1}|\partial_t\overline{v}|^2\leq +Cs\lambda^2\iint_{\overline{Q}}\overline{\xi} |\nabla \overline{v}|^2\,dx\,dt+Cs\lambda^4\iint_{Q} \overline{\xi} \overline{v}^2\,dx\,dt+\|\overline{P}_1\overline{v}\|_{L^2(\overline{Q})}^2,
	\end{split}
\end{equation*}
and
\begin{equation*}
	\begin{split}
		s^{-1}\iint_{\overline{Q}}\xi^{-1}|\Div(|x|^\alpha \nabla \overline{v})|^2\,dx\,dt\leq Cs^3\lambda^4\iint_{\overline{Q}}\overline{\xi}^3\overline{v}^2\,dx\,dt+Cs\iint_{\overline{Q}}\overline{\xi}^3 \overline{v}^2\,dx\,dt+\|\overline{P}_2\overline{v}\|_{L^2(\overline{Q})}^2.
	\end{split}
\end{equation*}
This together with \eqref{12.24.5} gives
\begin{equation}\label{12.24.6}
	\begin{split}
		&+\iint_{\overline{Q}} \left[s^{-1}\overline{\xi}^{-1}\left(|\partial_t \overline{v}|^2+|\Div(|x|^\alpha \nabla \overline{v})|^2\right)+s\lambda^2 \overline{\xi} |\nabla \overline{v}|^2+s^3\lambda^4 \xi^3 \overline{v}^2\right]\,dx\,dt\\
		&\leq +C\iint_{\overline{Q}}e^{-2s\overline{\sigma}}g^2\,dx\,dt+Cs\lambda^2\iint_{\widehat{\omega}\times (0,T)}\overline{\xi} |\nabla \overline{v}|^2\,dx\,dt+Cs^3\lambda^4\iint_{\widehat{\omega} \times (0,T)} \overline{\xi}^3|\overline{v}|^2\,dx\,dt.
	\end{split}
\end{equation}

Choose $\rho \in C^\infty(\mathbb{R}^N), 0 \leq \rho \leq 1$, satisfying
\begin{equation*}
	\rho = 1   \text{ on } A_{4R,5R},   \rho = 0   \text{ on } \mathbb{R}^N \setminus B_{6R}.
\end{equation*}
Then,
\begin{eqnarray*}
	&&+s\lambda^2 \iint_{\Omega \times (0,T)} \overline{\xi} \rho^2 |x|^\alpha |\nabla \overline{v}|^2 \, dx \, dt \\
	&&= -s\lambda^2 \iint_{\Omega \times (0,T)} \overline{\xi} \rho^2 \Div(|x|^\alpha \nabla \overline{v}) \overline{v} \, dx \, dt - s\lambda^3 \iint_{\Omega \times (0,T)} \overline{\xi} \rho^2 v |x|^\alpha \nabla \overline{v} \cdot \nabla \overline{\eta} \, dx \, dt \\
	&&\hspace{4.5mm} - 2s\lambda^2 \iint_{\Omega \times (0,T)} \overline{\xi} \rho v |x|^\alpha \nabla \overline{v} \cdot \nabla \rho \, dx \, dt \\
	&&\leq +\varepsilon s^{-1} \iint_{\Omega \times (0,T)} \overline{\xi}^{-1} |\Div(|x|^\alpha \nabla \overline{v})|^2 \, dx \, dt + \varepsilon s\lambda^2 \iint_{\Omega \times (0,T)} \overline{\xi} \rho^2 |x|^\alpha |\nabla \overline{v}|^2 \, dx \, dt \\
	&&\hspace{4.5mm} + C(\varepsilon)s^3\lambda^4 \iint_{\Omega \times (0,T)} \overline{\xi}^3 \overline{v}^2 \, dx \, dt + C(\varepsilon)s\lambda^4 \iint_{\Omega \times (0,T)} \overline{\xi} \overline{v}^2 \, dx \, dt.
\end{eqnarray*}
Taking $\varepsilon > 0$ sufficiently small, and considering
\begin{equation*}
	s\lambda^2 \iint_{\widehat{\Omega} \times (0,T)} \overline{\xi} |\nabla \overline{v}|^2 \, dx \, dt \leq Cs\lambda^2 \iint_{\Omega \times (0,T)} \overline{\xi} \rho^2 |x|^2 |\nabla \overline{v}|^2 \, dx \, dt,
\end{equation*}
we obtain
\begin{equation}\label{12.24.7}
	\begin{split}
		&\iint_{\overline{Q}} \left[ s^{-1}\overline{\xi}^{-1} \left( |\partial_t \overline{v}|^2 + |\Div(|x|^\alpha \nabla \overline{v})|^2 \right) + s\lambda^2 \overline{\xi} |\nabla \overline{v}|^2 + s^3\lambda^4 \xi^3 \overline{v}^2 \right] \, dx \, dt \\
		&\leq C\iint_{\overline{Q}} e^{-2s\overline{\sigma}} g^2 \, dx \, dt + Cs^3\lambda^4 \iint_{\widehat{\Omega} \times (0,T)} \overline{\xi}^3 |\overline{v}|^2 \, dx \, dt.
	\end{split}
\end{equation}
This follows from \eqref{12.24.6}.
Next, using $\nabla \overline{\varphi}_0 = e^{s\overline{\sigma}} (\nabla \overline{v} - s\lambda \overline{\xi} \overline{v} \nabla \overline{\eta})$, we find
\begin{equation*}
	\begin{split}
		&s\lambda^2 \iint_{\overline{Q}} e^{-2s\overline{\sigma}} \overline{\xi} |\nabla \overline{\varphi}_0|^2 \, dx \, dt + s^3\lambda^4 \iint_{\overline{Q}} e^{-2s\overline{\sigma}} \overline{\xi}^3 \overline{\varphi}_0^2 \, dx \, dt \\
		&\leq C\iint_{\overline{Q}} e^{-2s\overline{\sigma}} g^2 \, dx \, dt + Cs^3\lambda^4 \iint_{\Omega \times (0,T)} e^{-2s\overline{\sigma}} \overline{\xi}^3 \overline{\varphi}_0^2 \, dx \, dt.
	\end{split}
\end{equation*}
From this, we derive the following Theorem \ref{12.24.T1}.
\begin{theorem}\label{12.24.T1}
	Let $\overline{\varphi}_0$ be a solution of \eqref{12.24.1}. Then, there exist constants $\lambda_0 \geq 1, s_0 \geq 1$, and $C > 0$ such that for all $\lambda \geq \lambda_0$ and for any $s \geq s_0$, we have
	\begin{equation}\label{12.24.8}
		\begin{split}
			&s\lambda^2 \iint_{\overline{Q}} e^{-2s\overline{\sigma}} \overline{\xi} |\nabla \overline{\varphi}_0|^2 \, dx \, dt + s^3\lambda^4 \iint_{\overline{Q}} e^{-2s\overline{\sigma}} \overline{\xi}^3 \overline{\varphi}_0^2 \, dx \, dt \\
			&\leq C\iint_{\overline{Q}} e^{-2s\overline{\sigma}} g^2 \, dx \, dt + Cs^3\lambda^4 \iint_{\Omega \times (0,T)} e^{-2s\overline{\sigma}} \overline{\xi}^3 \overline{\varphi}_0^2 \, dx \, dt.
		\end{split}
	\end{equation}
\end{theorem}
By Theorem \ref{12.24.T1}, and considering the definitions of $\Omega = A_{3R,6R}$ and $\overline{\varphi}_0 = \widehat{\varphi}_0$ on $(\Omega \setminus B_{5R}) \times (0,T)$, we have
\begin{eqnarray*}
	&&s\lambda^2 \iint_{(\Omega \setminus B_{5R}) \times (0,T)} e^{-2s\overline{\sigma}} \overline{\xi} |\nabla \widehat{\varphi}_0|^2 \, dx \, dt + s^3\lambda^4 \iint_{(\Omega \setminus B_{5R}) \times (0,T)} e^{-2s\overline{\sigma}} \overline{\xi}^3 \widehat{\varphi}_0^2 \, dx \, dt \\
	&&\leq C\iint_{(\Omega \setminus B_R) \times (0,T)} e^{-2s\overline{\sigma}} \left( 2w\nabla\kappa \cdot \nabla \widehat{\varphi}_0 + \widehat{\varphi}_0 \Div(w\nabla\kappa) \right) \, dx \, dt \\
	&&\quad + Cs^3\lambda^4 \iint_{(B_{6R} \setminus B_{3R}) \times (0,T)} e^{-2s\overline{\sigma}} \overline{\xi}^3 \overline{\varphi}_0^2 \, dx \, dt \\
	&&\leq C\iint_{(B_{5R} \setminus B_{4R}) \times (0,T)} e^{-2s\overline{\sigma}} \left( \widehat{\varphi}_0^2 + |\nabla \widehat{\varphi}_0|^2 \right) \, dx \, dt + Cs^3\lambda^4 \iint_{(B_{6R} \setminus B_{3R}) \times (0,T)} e^{-2s\overline{\sigma}} \overline{\xi}^3 \widehat{\varphi}_0^2 \, dx \, dt \\
	&&\leq Cs^3\lambda^4 \iint_{(B_{6R} \setminus B_{3R}) \times (0,T)} e^{-2s\overline{\sigma}} \overline{\xi}^3 \widehat{\varphi}_0^2 \, dx \, dt,
\end{eqnarray*}
where the last inequality employs a Caccioppoli-type inequality similar to Lemma \ref{12.23.L1}.
Finally, by the definition of $g$, and utilizing the inequalities
\begin{equation*}
	\Theta^3e^{-2s\overline{\sigma}} \geq C > 0   \text{ in } \Omega \times \left(\frac{T}{4}, \frac{3T}{4}\right)
\end{equation*}
and
\begin{equation*}
	\Theta^3e^{-2s\overline{\sigma}} \leq C   \text{ in } Q,
\end{equation*}
where the constants $C > 0$ depend solely on $\alpha, R, N, T$, and $\Omega$, we obtain
\begin{equation*}\label{12.22.3}
	\begin{split}
		&\iint_{(\Omega \setminus B_{6R}) \times \left(\frac{T}{4}, \frac{3T}{4}\right)} |\nabla\widehat{\varphi}_0|^2 \, dx \, dt + \iint_{(\Omega \setminus B_{6R}) \times \left(\frac{T}{4}, \frac{3T}{4}\right)} \widehat{\varphi}_0^2 \, dx \, dt \leq C \iint_{(B_{6R} \setminus B_{4R}) \times (0,T)} \widehat{\varphi}_0^2 \, dx \, dt.
	\end{split}
\end{equation*}

This, together with the definition of $\psi$, leads to the following Theorem \ref{04.18.T1g}.

\end{document}